\documentclass[12pt,a4paper,twoside]{article}

\usepackage{amsmath}
\usepackage{amssymb}
\usepackage{amsthm}
\usepackage[utf8]{inputenc}
\usepackage[ruled,vlined]{algorithm2e}
\usepackage{graphicx}
\usepackage[justification=centering]{caption}
\usepackage{subfig}
\usepackage[export]{adjustbox}
\usepackage{setspace}
\usepackage{layout}
\usepackage{lineno,soul}
\usepackage[dvipsnames]{xcolor}
\PassOptionsToPackage{hyphens}{url}
\usepackage[citecolor=NavyBlue,colorlinks=true,linkcolor=NavyBlue,urlcolor=NavyBlue,breaklinks]{hyperref}
\usepackage[hyphenbreaks]{breakurl}
\usepackage{cleveref}
\usepackage[title]{appendix}
\usepackage{bm}
\usepackage{booktabs}
\usepackage[numbib,notlof,notlot,nottoc]{tocbibind}
\newcommand{\norm}[1]{\left\lVert #1\right\rVert}

\newcommand\bb[1]{\mathbb{#1}} % mathbb
\newcommand\q[1]{\mathcal{#1}} % math cursive
\newcommand\conj[1]{\overline{#1}}

\newcommand\RE{\operatorname{Re}}
\newcommand\IM{\operatorname{Im}}

\newcommand\MOD{\operatorname{mod}}

\newcommand\diag{\operatorname{diag}}

\newtheorem{theorem}{Theorem}

\newtheorem{lemma}[theorem]{Lemma}

\newtheorem{remark}[theorem]{Remark}
\hypersetup{%
    bookmarksnumbered, bookmarksopen=true, bookmarksopenlevel=1,%
}

\numberwithin{theorem}{section} 

\DeclareMathAlphabet{\myds}{U}{bbold}{m}{n}

\title{Convergence properties of gradient methods for blind ptychography}
\author{Oleh Melnyk\thanks{Mathematical Imaging and Data Analysis, Helmholtz Munich, 85764 Neuherberg, Germany (\href{mailto:oleh.melnyk@helmholtz-munich.de}{oleh.melnyk@helmholtz-munich.de}).} }

\begin{document}

%\frontmatter
\maketitle
\begin{abstract}
We consider blind ptychography, an imaging technique which aims to reconstruct an object of interest from a set of its diffraction patterns, each obtained by a local illumination. As the distribution of the light within the illuminated region, called the window, is unknown, it also has to be estimated as well. For the recovery, we consider gradient and stochastic gradient descent methods for the minimization of amplitude-base squared loss. In particular, this includes extended Ptychographic Iterative Engine as a special case of stochastic gradient descent. We show that all methods converge to a critical point at a sublinear rate with a proper choice of step sizes. We also discuss possibilities for larger step sizes.

\textbf{Keywords:} blind ptychography, gradient descent, stochastic gradient descent, extended Ptychographic Iterative Engine.

\textbf{MSC Codes:}  78A46, 78M50, 47J25, 90C26.
\end{abstract}

\section{Introduction}\label{sec: introduction}

In recent years, ptychography \cite{Hoppe.1969} became a prominent technique in diffractive imaging. Instead of a single illumination of the object of interest, multiple local illuminations of the different parts of the object are performed and for each the diffraction pattern is captured by a detector placed in the far field. Consequently, the object must be reconstructed from the obtained diffraction patterns. Furthermore, it is often the case that the distribution of light called the probe or the window, is also not known and has to be estimated as well. In this case, the corresponding recovery problem is sometimes referred to as blind ptychography. 

The locality of illuminations allows to achieve better resolution and recently even reach a sub-{\AA}ngstr{\"o}m scale \cite{Chen.2021}. This led to a rise in popularity among practitioners and its numerous applications in fields such as crystallography \cite{Lazarev.2018, Chen.2021}, biology \cite{Piazza.2014, Giewekemeyer.2015} and material science \cite{Hoydalsvik.2014, Esmaeili.2015}. At the same time, an interest in ptychography sparked the development of efficient reconstruction methods. One proposed technique, Ptychographic Iterative Engine (PIE) \cite{Rodenburg.2004} and its extended version (ePIE) for blind ptychography \cite{Maiden.2009, Maiden.2017} became especially popular among practitioners. Over the years, it has been adapted for other measurement scenarios such as multislice objects \cite{Maiden.2012}, multimodal illumination \cite{Batey.2014} and tomographic ptychography \cite{Li.2016}. An apparent reason for ePIE's success is its simple explanation, implementation and computational complexity as the algorithm iteratively updated the object and the probe by processing a single diffraction at the time. However, ePIE is still not supported by a mathematical convergence analysis, which makes its success seem like a miracle. 

On the other hand, there has been a lot of progress in understanding ptychography as an inverse problem. When the window is known, it can be seen as a special case of phase retrieval with the short-time Fourier transform measurements \cite{Grohs.2020, Bendory.2017, Alaifari.2021}. In the past decade, there has been a series of studies on the uniqueness and stability of ptychographic phase retrieval both in continuous \cite{Grohs.2020, Grohs.2022, Grohs.2023} and discrete cases \cite{Bojarovska.2016, Li.2017, Bendory.2017, Bendory.2018, Alaifari.2021}, i.e., the object is a function or a vector, respectively. Most of these works to some extent depend on the Wigner distribution deconvolution \cite{Cohen.1989, Rodenburg.1992, Chapman.1996}, the representation of the ptychographic measurements as a convolution of the Wigner transform of the object convolved with the Wigner transform of the window. This procedure leads to the constructive noniterative procedure for the reconstruction of the object \cite{Chapman.1996, Iwen.2016, Iwen.2020, Perlmutter.2020, Forstner.2020, Melnyk.2023}. Besides, there are many iterative algorithms \cite{Fienup.1978, Elser.2003, Luke.2005, Candes.2013, Ghods.2018, Fatima.2022, Li.2022}, which can be applied to solve phase retrieval. The most prominent group is algorithms that pose the recovery as an optimization problem, and among them - are gradient methods. It includes famous Error Reduction \cite{Gerchberg.1972, Marchesini.2016, Melnyk.2022}, Wirtinger \cite{Candes.2015, Chen.2023} and Amplitude \cite{Wang.2018, Xu.2018} Flows and the above-mentioned PIE, which can be seen as the stochastic gradient descent \cite{Melnyk.2022b}.

In contrast, blind ptychography is much less discussed in the literature. In \cite{Fannjiang.2020b} the authors study the ambiguities resulting from regular scanning grids and search for possible cures. The uniqueness of reconstruction can be guaranteed for generic objects and windows \cite{Bendory.2020, Bendory.2022}. With few exceptions, reconstruction algorithms for blind ptychography exploit an idea of alternating minimization \cite{Hesse.2015, Chang.2019, Fannjiang.2020, Filbir.2022}. That is the object or the window is fixed, while the other unknown is being optimized and after some iterations the fixed and optimized unknowns are swapped. In this way, the recovery problem reduces to consequent phase retrieval problems, which are well-understood. In \cite{Roach.2023}, Wigner distribution deconvolution is combined with blind deconvolution. Some algorithms \cite{Maiden.2009, Maiden.2017, Odstrcil.2018, Kandel.2021} perform joint optimization of the object and the window, however, the convergence of such methods is not guaranteed. 
  
In this paper, we study gradient methods for joint optimization of the object and the window. We show that with carefully selected step sizes gradient descent and stochastic gradient descent converge to a critical point of the loss function at a sublinear rate. The second result is particularly valuable in the context of ePIE, which, just as PIE, can be seen as a stochastic gradient descent. To our knowledge, this is the first result in the literature regarding the convergence of ePIE. 

We start with preliminaries in \Cref{sec: preliminaries}. That includes a mathematical description of the measurements and the choice of the loss function in \Cref{sec: model} as well as an overview of ePIE and its connection to the stochastic gradient \Cref{sec: epie}. Then, our results are presented in \Cref{sec: results} and we also discuss the possibility of larger step sizes in \Cref{sec: larger step size}. All proofs can be found in \Cref{sec: proofs} followed by a short conclusion. 

\section{Preliminaries}\label{sec: preliminaries}

\subsection{Notation}
The set of $d$-dimensional complex vectors is denoted by $\bb C^d$ and $d_1 \times d_2$ matrices by $\bb C^{d_1 \times d_2}$. For a complex number $z$, its real, imaginary parts and absolute value are denoted by $\RE z$, $\IM z$, $|z|$, respectively. For $v \in \bb C^d$ let $\conj v$, $v^T$ and $v^*$ denote the complex conjugate, transpose and complex conjugate transpose, respectively. The entrywise (Hadamard) product is denoted by $\circ$. Notation $\norm{v}_p$, $1 \le p \le \infty$ denotes the $\ell_p$-norm of the vector $v \in \bb C^d$.

For a vector $v \in \bb C^d$, the diagonal matrix $\diag(v) \in \bb C^{d \times d}$ contains entries $\diag(v)_{j,j} = v_j$ on the main diagonal and zeros elsewhere. Let $F$ be the discrete Fourier transform matrix with entries $F_{j,k} = e^{-2\pi i j k /d}$. The shift operator $S_r,$ $r \in \bb Z$ can be treated either as circular shift $(S_r v)_j = v_{j - r \MOD d}$ or as noncircular shift $(S_r v)_j = v_{j - r}$ for $0 \le j - r \le d$ and $(S_r v)_j =0$ otherwise. The only equation, where its definition is used explicitly is \eqref{eq: bilinear norm bound} and the bound holds in both cases. The identity matrix is $I \in \bb C^{d \times d}$. For a matrix $A$, $\norm{A}$ denotes its spectral norm, $\norm{A} = \max_{\norm{v}_2 = 1} \norm{Av}_2$.

We set $\min\{a, b / 0 \} =a$ for the step sizes below, e.g., in \eqref{eq: mu max}.

For a function $f : \bb C^d \times \bb C^d \to \bb R:$ $(z,v) \mapsto f(z,v)$, the notation $\nabla f = (\nabla_z f, \nabla_v f)$ denotes its generalized Wirtinger gradient. While this superficial definition is sufficient for understanding our results below, we also provide a proper definition and selected properties in \Cref{sec: Wirtinger derivatives}. 

\subsection{Mathematical model, losses and gradients}\label{sec: model}

To describe the ptychographic measurements mathematically, let us introduce the notation. We denote the object and the window as $x \in \bb C^d$ and $w \in \bb C^d$, respectively. Each of the illuminated regions corresponds to a shift of the window $S_r w$ by some $r$ and the set of all shifts is $\q R$ with $R$ elements. The interaction between the object and shifted window is modeled by entrywise multiplication, resulting in the exit wave $x \circ S_r w$. The propagation of the exit wave to the far field corresponds to an application of the discrete Fourier transform $F$. Finally, the intensity of the captured wave is given by a square of its absolute value. The resulting measured noisy intensities are described by the following equation
\begin{equation}\label{eq: meas}
y_{r,k} = Noisy(|[F(x \circ S_r w)]_k|^2), \quad r \in \q R, \quad k=1,\ldots, d.
\end{equation}
Consequently, our goal is to recover $x$ and $w$ from the measurements $y_{r,k}$. 
We note that unique recovery is understood up to several ambiguities naturally arising from the measurements \eqref{eq: meas}. 

\begin{theorem}[{General ambiguities arising in blind ptychography \cite{Fannjiang.2020b}}] \label{thm: blind ambiguities}
Consider $x,w \in \bb C^{d}$ and the corresponding ptychographic measurements \eqref{eq: meas}. Then,  
\begin{enumerate}
\item (global phase ambiguity) for all $a, b$ such that $|a| = |b|=1$ the pair $a x, b w$ produces the same measurements \eqref{eq: meas},
\item (linear phase ambiguity) for all $\rho \in \bb R$ the pair $z,v \in \bb C^d$ with $z_k = e^{-i \rho k} x_k$ and $v_k = e^{i \rho k} w_k$. $k =1,\ldots, d$, produces the same measurements \eqref{eq: meas},
\item (scaling ambiguity) for all $\gamma \in \bb C \backslash \{0\}$ the pair $\gamma x, w / \gamma$ produces the same measurements \eqref{eq: meas}.
\end{enumerate}
\end{theorem}
Furthermore, other ambiguities may arise depending on the set of shifts $\q R$, e.g., when $\q R$ forms a regular grid \cite{Bendory.2020}.  

In order to pose the recovery problem as an optimization problem we consider the loss 
\begin{equation}\label{eq: loss}
\q L_\varepsilon(z,v) = \sum_{r \in \q R} \sum_{k = 1}^d \left[ \sqrt{ | [ F (z \circ S_r v) ]_k |^2 + \varepsilon} - \sqrt{y_{r,k} + \varepsilon} \right]^2, \quad \varepsilon \ge 0.
\end{equation}
If the measurements are noiseless, we get $\q L_\varepsilon(x,w) = 0$ and $\q L_\varepsilon(z,v) \ge 0$, so that the pair $x,w$ is a global minimum of $\q L_\varepsilon$. However, as the loss function \eqref{eq: loss} is nonconvex, its global optimization is a nontrivial task. Moreover, we would only consider a convergence to a critical point $\nabla \q L_{\varepsilon} = 0$, with the gradients of \eqref{eq: loss} given by
\begin{align*}
\nabla_z \q L_\varepsilon (z,v) & :=  \sum_{r \in \q R} \diag( S_{r} \conj v) F^* \left[ I -  \diag\left( \left\{ \frac{ \sqrt{y_{r,k} + \varepsilon}}{ \sqrt{| [F (z \circ S_r v) ]_{k} |^2 + \varepsilon}} \right\}_{k=1}^{d} \right) \right] F (z \circ S_r v), \\
\nabla_v  \q L_\varepsilon (z,v) & :=  \sum_{r \in \q R} \diag(\conj z) S_{-r} F^* \left[ I -  \diag\left( \left\{ \frac{ \sqrt{y_{r,k} + \varepsilon}}{ \sqrt{| [F (z \circ S_r v) ]_{k} |^2 + \varepsilon}} \right\}_{k=1}^{d} \right) \right] F (z \circ S_r v),
\end{align*}
where in the case $\varepsilon = 0$ and $F (z \circ S_r v)_k = 0$ the fraction $F (z \circ S_r v)/|F (z \circ S_r v)|$ is set to zero.

Intuitively, the squared error between the simulated and measured intensities would be the first choice, i.e., without additional square roots as in \eqref{eq: loss}. Our choice is motivated by the three following facts. The function \eqref{eq: loss} with $\varepsilon = 0$ is a second-order Taylor approximation of the maximum likelihood function for Poisson noise \cite{Thibault.2012, Romer.2022}. %\OM{Patricia}. 
The minimization of $\q L_0$ is also equivalent to finding a barycenter of positive-semidefinite matrices with respect to the Bures-Wasserstein distance \cite{Maunu.2022}. Lastly, the partial gradients $\nabla_z \q L_\varepsilon$ and $\nabla_v \q L_\varepsilon$ admit Lipschitz-like properties, which is essential for the analysis of the convergence of the gradient methods \cite{Hesse.2015, Xu.2018, Melnyk.2022, Melnyk.2022b}. A nonzero smoothing parameter $\varepsilon$ makes the gradient Lipschitz continuous, which is mostly needed as a technical step in the analysis \cite{Xu.2018, Chang.2019, Melnyk.2022}, but it also has an impact on the convergence to a critical point for stochastic gradient descent \cite{Melnyk.2022b}.

Note the loss function \eqref{eq: loss} is prone to the ambiguities described in \Cref{thm: blind ambiguities}. The first two ambiguities only affect the phases of the object and the window and the scaling ambiguity also affects their norms. As the gradients of $\q L_\varepsilon$ are not scale invariant, it is possible to make $\nabla_z \q L_\varepsilon$ arbitrarily large and $\nabla_v \q L_\varepsilon$ proportionally small and vice versa, while the loss function remains constant. This later complicates the analysis of the convergence. A possible cure to the scaling problem is to impose a constraint on the norm of $z$ and $v$. This can be done explicitly by adding constraints to the optimization problem, for instance, see \cite{Hesse.2015} or Section 4.2.4 of \cite{Melnyk.2023}. An alternative is to impose the norm constraint implicitly by including the Tikhonov regularization,      
\begin{equation}\label{eq: loss regularized}
\q J(z,v)  := \q J(z,v; \varepsilon, \alpha_T, \beta_T) = \q L_\varepsilon(z,v) + \alpha_T \norm{z}_2^2 + \beta_T \norm{v}_2^2,
\end{equation}
for some $\alpha_T, \beta_T \ge 0$ with gradients
\[
\nabla_z \q J(z,v) = \nabla_z \q L_\varepsilon(z,v) + \alpha_T z, \quad \nabla_v \q J(z,v) = \nabla_v \q L_\varepsilon + \beta_T v.
\] 
It was shown that the norms of the iterates of Alternating Amplitude Flow for $\q J$ \cite{Filbir.2022} are bounded and we will use a similar argument in our proofs later.
 
For the minimization of \eqref{eq: loss regularized}, we employ the gradient descent scheme
\begin{equation}\label{eq: gradient descent}
z^{t+1} = z^t - \mu_t \nabla_z \q J(z^t,v^t),
\quad \quad
v^{t+1} = v^t - \nu_t \nabla_v \q J(z^t,v^t),
\end{equation}
with initial guesses $z^0,v^0$ and step sizes $\mu_t,\nu_t$. While it is simple and not the most efficient optimization method, understanding its convergence is the first step toward the analysis of more involved methods such as \cite{Kandel.2021} or ePIE.  

\subsection{Extended Ptychographic Iterative Engine}\label{sec: epie}

Extended Ptychographic Iterative Engine (ePIE) is an iterative algorithm which extends the idea of Error Reduction \cite{Gerchberg.1972} for blind ptychography. Starting with initial guesses $z^0, v^0$, to construct the next iterates for an iteration it selects a single region $r^t \in \q R$ and computes the exit wave $z^t \circ S_{r^t} v^t$. Then, a single Error Reduction iteration is performed, which gives a corrected exit wave. Finally, new object $z^{t+1}$ and the window $v^{t+1}$ are decoupled from the two exit waves. The detailed iteration is summarized in \Cref{alg: ePIE iteration}.

\medskip
\begin{algorithm}[H]
\caption{ePIE iteration, version of \cite{Maiden.2009}}
\label{alg: ePIE iteration}
\SetAlgoLined
\SetKwInOut{Input}{Input}
%\SetKwInOut{Output}{Output}
\Input{Measurements $y$, previous iterates iterates $z^t,v^t \in \bb C^d$, step sizes $\alpha_t,\beta_t > 0$.}
%\Output{$z^{t+1} \in \bb C^d$, $\hat v^{t+1} \in \bb C^\delta$.}
1. Select a shift position $r^t \in \q R$. \\
2. Construct an exit wave $\psi =  z^t \circ S_{r^t} v^t$.\\
3. Compute its Fourier transform $\Psi = F \psi$.\\
4. Correct the magnitudes of $\Psi$ as $\Psi'_k = \sqrt{ y_{r^t, k} } \Psi_k /|\Psi_k|$.\\
5. Find an exit wave $\psi'$ corresponding to $\Psi'$ via $\psi' = F^{-1} \Psi' $ \\
6. Return  
\begin{align*}
z^{t+1} = z^t + \frac{\alpha_t \diag(S_{r^t} \conj{v}^t)}{\norm{v^t}_\infty^2}  [\psi' - \psi], & \quad 
v^{t+1} = v^t + \frac{\beta_t S_{-r^t} \diag(\conj{z}^t)}{\norm{z^t}_\infty^2}   [\psi' - \psi].
\end{align*}
\\
\end{algorithm}
\medskip   
Parameters $\alpha_t$ and $\beta_t$ are often set to a constant value for all iterations, e.g., $\alpha_t = \beta_t = 0.05$. There are two common ways to choose $r^t$, either as a neighboring position to $r^{t-1}$ or it loops through the set $\q R$, which is randomly shuffled after each loop. 

As it follows from \Cref{alg: ePIE iteration}, an iteration of ePIE only performs two Fourier transforms and accesses one diffraction pattern $\{ y_{r^t, k} \}_{k=1}^d$, which makes it extremely efficient.

There are several interpretations of ePIE in the literature \cite{Thibault.2013, Maiden.2017} and the one in  \cite[Section 4.1.2]{Hesse.2015}, \cite[Section F.1]{Kandel.2021} or \cite[Section 4.3]{Melnyk.2023} is particularly useful and links ePIE to the function $\q L_\varepsilon$. For this, let us consider supplementary error functions
\[
\q L_{r,\varepsilon} := \sum_{k = 1}^d \left[ \sqrt{ | [ F (z \circ S_r v) ]_k |^2 + \varepsilon} - \sqrt{y_{r,k} + \varepsilon} \right]^2, \quad r \in \q R,
\]
corresponding to the squared error for $r$-th illuminated region. Then, the ePIE iteration can be interpreted as a gradient step 
\begin{equation}\label{eq: ePIE as selection of objective}
z^{t+1} = z^{t} - \frac{\alpha_t}{d \norm{v^t}_\infty^2} \nabla_{z} \q L_{r^t,0}(z^t,v^t), \quad 
v^{t+1} = v^{t} - \frac{\beta_t}{d \norm{z^t}_\infty^2} \nabla_{v} \q L_{r^t,0}(z^t,v^t).
\end{equation}
While representation \eqref{eq: ePIE as selection of objective} connects ePIE to $\q L_0$, in this form is not convenient to analyze ePIE as a gradient method for $\q L_0$. 

Let us assume that that each $r^t$ is sampled independently at random from some distribution $p \in (0,1)^R$ over $\q R$. Then, we are able to reinterpret ePIE as a stochastic gradient descent for $\q L_{0}$. First, let us formally introduce the stochastic gradient. 

Consider a function $f: \bb C^{d} \times \bb C^d \to \bb R$ with decomposition $f = \sum_{r \in \q R} f_r$ and generalized Wirtinger gradients $\nabla f_r$. Fix $K \in \bb N$ to be the number of summands to be sampled and let $p \in (0,1)^R$ be a distribution over $\q R$. Then, the stochastic gradient $g[f](z,v)$ is constructed as
\begin{equation}\label{eq: stochastic gradient}
g[f] (z,v) := \frac{1}{K} \sum_{k=1}^K \frac{1}{p_{r^k} } \nabla f_{r^k}(z,v)
\quad \text{with} \quad
g[f](z,v) = 
\begin{bmatrix}
g[f]_z (z,v) \\
g[f]_v (z,v)
\end{bmatrix},
\end{equation}
where indices $r^{1}, \ldots, r^{K}$ are sampled independently with replacement from $p$. When working with iterations, we would write $g^t$ with independent indices $r^{t,1}, \ldots, r^{t,K}$ to indicate the dependence on $t$. Note that for $K = 1$ the equation \eqref{eq: stochastic gradient} simplifies to
\[
g[f](z,v) = \frac{1}{p_{r}} \nabla f_{r}(z,v).
\]
for a sampled index $r$. Consequently, we can rewrite \eqref{eq: ePIE as selection of objective} as 
\begin{equation}\label{eq: ePIE as stochastic gradient}
z^{t+1} = z^{t} - \frac{\alpha_t p_{r^t}}{d \norm{v^t}_\infty^2} g_{z}^t[\q L_0](z^t,v^t), \quad 
v^{t+1} = v^{t} - \frac{\beta_t p_{r^t}}{d \norm{z^t}_\infty^2} g_{v}^t[\q L_0](z^t,v^t).
\end{equation}
Thus, ePIE can be understood and analyzed as a stochastic gradient descent for $\q L_0$. This motivated us to study the convergence properties of stochastic gradient descent for $\q J$,
\begin{equation}\label{eq: stochastic gradient descent}
z^{t+1} = z^t - \mu_t g_z^t[\q J](z^t,v^t),
\quad \quad
v^{t+1} = v^t - \nu_t g_v^t[\q J](z^t,v^t),
\end{equation} 
were the decomposition of $\q J = \sum_{r \in \q R} \q J_{r}$ is given by
\begin{equation}\label{eq: Jr definition}
\q J_r(z,v;\varepsilon,\alpha_T, \beta_T) := \q L_{r,\varepsilon}(z,v)  + \alpha_T p_r \norm{z}_2^2 + \beta_T p_r \norm{v}_2^2.
\end{equation}
From now on, we would use a shorter notation $g$ instead of $g[\q J]$.

As the stochastic gradients are random, their convergence can be only guaranteed in probabilistic terms. Let us briefly introduce a few notions from probability theory, while more details on the topic can be found in \cite{Athreya.2006}. Consider the probability space $(\q S, \q F, \bb P)$ induced by the random indices $\{r^{t,k} \}_{t = 0, k= 1}^{\infty, K}$. Let $\q F_0$ be the trivial $\sigma$-algebra and for $t \ge 1$ let $\q F_t$ be the $\sigma$-algebra generated by $\{r^{s,k}\}_{s = 0, k=1}^{t-1,K}$. By construction, $\q F_0 \subseteq \q F_1 \subseteq \ldots \q F$ so that $\{\q F_t \}_{t \ge 0}$ is a filtration. We say that a sequence of random variables $\{ X_t \}_{t \ge 0}$ is adapted to filtration $\{\q F_t \}_{t \ge 0}$ if for each $t$ $X_t$ is $\q F_t$ measurable. Essentially, all random variables which are $\q F_t$-measurable can be expressed in terms of $\{r^{s,k}\}_{s = 0, k=1}^{t-1,K}$. In particular, $z^t, v^t$ are $\q F_t$-measurable and, consequently, $\q J(z^t, v^t)$ and $\norm{\nabla \q J(z^t, v^t)}_2$ are adapted to the filtration $\{\q F_t \}_{t \ge 0}$. Note that $\{ g(z^t,v^t) \}_{t \ge 0}$ is not adapted as it also depends on $\{r^{t,k}\}_{k=1}^{K}$. We say that a random event is almost surely (a.s.) true if its probability is one. We would denote expectation by $\bb E$ and conditional expectation with respect to $\sigma$-algebra $\q F_t$ by $\bb E[\, \cdot \,| \, \q F_t \, ]$. With this, we are ready to state our results.

\section{Our contribution}\label{sec: results}
%\subsection{Descent lemma and gradient descent}\label{sec: gradient descent}
Let us start first with the analysis of the gradient descent \eqref{eq: gradient descent}. It was shown that the suitable minimal requirement for the convergence of gradient methods for nonconvex optimization is the so-called descent lemma \cite[Assumption 3]{Khaled.2023}, which controls the change in the loss function in terms of the change in the argument. In fact, it can be seen as a weaker version of Lipschitz continuity of the gradient. For $\q J$ we are able to establish the following version of the descent lemma  
\begin{lemma}[Descent lemma for $\q J$]\label{col: descent lemma J}
Let $\varepsilon, \alpha_T, \beta_T \ge 0$. For all $z,v \in \bb C^d$ and $u,h \in \bb C^d$ we have
\begin{align*}
\q J(z+u, v + h) 
& \le \q J(z, v) 
+ 2 \RE( u^* \nabla_z \q J(z,v) ) 
+ 2 \RE(h^*\nabla_v \q J(z,v)) \\
& \quad + \norm{u}_2^2  \left[  \alpha_T + d \left(\tfrac{10}{3} \norm{v}_2^2 + \tfrac{5}{4} \norm{h}_2^2 + \tfrac{2}{3} \norm{z}_2^2  +  \tfrac{1}{4}\norm{u}_2^2 + \norm{y/d}_1^{1/2} \right) \right] \\
& \quad + \norm{h}_2^2  \left[ \beta_T + d \left( \tfrac{10}{3} \norm{z}_2^2 + \tfrac{5}{4} \norm{u}_2^2 + \tfrac{2}{3} \norm{v}_2^2 +  \tfrac{1}{4}\norm{h}_2^2 + \norm{y/d}_1^{1/2} \right) \right]. 
%& \quad + d ( \norm{u}_2^2 + \norm{h}_2^2 ) \left[ \tfrac{10}{3} (\norm{z}_2^2 + \norm{v}_2^2) + \tfrac{5}{4} (\norm{u}_2^2 + \norm{h}_2^2) + \norm{y + \varepsilon}_1^{1/2} \right]. 
\end{align*}
\end{lemma}   
Unlike the standard decent lemma \cite[Lemma 5.7]{Beck.2017}, we observe the fourth-order terms appearing in \Cref{col: descent lemma J}. This can be explained by viewing $\q J$ as an ``almost'' fourth-order polynomial. That is $z \circ S_r v$ is a quadratic polynomial and $\sqrt{ | [ F (z \circ S_r v) ]_k |^2 + \varepsilon}$ is an ``almost'' quadratic polynomial, which gives total quartic dependency on the arguments.

For further convenience, let us define
\begin{equation}\label{eq: B}
B(z,v) := 3 d \left[ \tfrac{10}{3} \norm{z}_2^2  + \tfrac{10}{3} \norm{v}_2^2 + \norm{y/d}_1^{1/2} \right] + 3 \max\{\alpha_T,\beta_T\}.
\end{equation}
and
\begin{equation}\label{eq: J inf}
\q J^{\inf} := \inf_{z,v \in \bb C^d} \q J(z,v).
\end{equation}

With \Cref{col: descent lemma J}, we are able to choose the step sizes $\mu_t,\nu_t$ such that the gradient descent provably converges.

\begin{theorem}[Convergence of the gradient descent for blind ptychography]\label{thm: convergence gradient}
Let $\varepsilon, \alpha_T, \beta_T \ge 0$. Consider a sequences $\{z^t\}_{t \ge 0}, \{v^t\}_{t \ge 0}$ generated by \eqref{eq: gradient descent} with arbitrary starting points $z^0, v^0 \in \bb C^d$ and the step sizes  
\begin{align} 
\mu_t, \nu_t & \le \min \left\{ B^{-1}(z^t,v^t),  (\tfrac{15}{4} d)^{-1/3} \norm{ \nabla_z \q J(z^t,v^t)}_2^{-2/3},  (\tfrac{15}{4} d)^{-1/3} \norm{\nabla_v \q J(z^t,v^t)}_2^{-2/3}  \right\},  \label{eq: step size gradient} %\\
%\nu_t & \le \min \left\{ \tfrac{1}{3} d^{-1}  \left[ \tfrac{10}{3} \norm{z^t}_2^2  + \tfrac{2}{3} \norm{v^t}_2^2 + \norm{y/d}_1^{1/2}  + \beta_T \right]^{-1}, \right. \nonumber \\
%& \quad \quad \quad  \left. \left( \tfrac{4}{3} \right)^{1/3} d^{-1/3} \norm{ \nabla_v \q J(z^t,v^t)}_2^{-2/3},  \tfrac{4^{1/3}}{5 \cdot 3^{1/3}} d^{-1/3} \norm{\nabla_z \q J(z^t,v^t)}_2^{-2/3}  \right\}. \nonumber
\end{align}
Then, we have
\[ 
\q J(z^{t+1}, v^{t+1}) 
\le \q J(z^t, v^t) 
- \mu_t \norm{ \nabla_z \q J(z^t,v^t) }_2^2 
- \nu_t \norm{ \nabla_v \q J(z^t,v^t)}_2^2, \quad t \ge 0, 
\]
If $\alpha_T, \beta_T >0$ and, additionally, $\mu_t$ and $\nu_t$ are precisely equal to the minimums in \eqref{eq: step size gradient}, then
\[
\norm{ \nabla \q J(z^t,v^t) }_2 \to 0 \text{ as } t \to \infty 
\]
and
\begin{align*}
\min_{t=0,\ldots, T-1}  \norm{ \nabla \q J(z^t,v^t) }_2^2
\le \max \left\{ C_1 T^{-1} [\q J(z^0, v^0) -  \q J^{\inf}], \left( C_1 T^{-1} [\q J(z^0, v^0) -  \q J^{\inf}] \right)^{3/2} \right\},
\end{align*}
where $C_1=C_1(z^0, v^0, d, y, \alpha_T, \beta_T) $ is given by
\begin{equation}\label{eq: c gradient}
C_1 := \max \left\{ d \left[ 20 (\alpha_T^{-1} + \beta_T^{-1}) \q J(z^0, v^0) + 6 \norm{y/d}_1^{1/2}\right]  + 2\max\{\alpha_T,\beta_T\} , (15 d)^{1/3} \right\}.
\end{equation}
\end{theorem}
In contrast to the nonblind case \cite{Xu.2018}, in which the step size is constant, in \Cref{thm: convergence gradient} the step sizes have to be chosen adaptively and depend on the norm of the gradient. This is a consequence of the much more volatile behavior of $\q J$ described in \Cref{col: descent lemma J}. Note that similar choices of step sizes with partial gradient normalization are necessary to guarantee convergence of the gradient descent for certain classes of loss functions \cite{Chen.2023}. 

%\subsection{Stochastic gradient descent}\label{sec: stochastic gradient descent}

With the convergence results for gradient descent established, we can turn to stochastic gradient descent \eqref{eq: stochastic gradient descent}. 
Using ideas from \cite{Liu.2022, Orabona.2020} and \cite{Melnyk.2022b}, we are able to establish a convergence result for scheme \eqref{eq: stochastic gradient descent}. To state the results, we introduce two new bounds.
\begin{lemma}\label{l: stochastic gradient properties 1}
Define
\begin{align}
B_z(z,v) & := (\tfrac{15}{4} d)^{1/2} \left[ \frac{d \norm{v}_2}{\sqrt K \min_{r \in \q R} p_r} \left( \norm{z}_2  \norm{v}_2 +  \norm{y/d}_1^{1/2} \right) + \alpha_T \norm{z}_2 \right], \label{eq: Bz}\\
B_v(z,v) & = (\tfrac{15}{4} d)^{1/2} \left[ \frac{d \norm{z}_2}{\sqrt K \min_{r \in \q R} p_r} \left( \norm{z}_2  \norm{v}_2 +  \norm{y/d}_1^{1/2} \right) + \beta_T \norm{v}_2 \right].  \nonumber
\end{align}
Then, for any $z,v\in \bb C^d$ we have
\[
\tfrac{15}{4} d \norm{g_z(z,v)}_2^2 \le B_z^2(z,v) 
\quad \text{and} \quad 
\tfrac{15}{4} d \norm{g_v(z,v)}_2^2 \le B_v^2(z,v).
\]
\end{lemma}
Now, we can state the main result regarding the convergence of the stochastic gradient descent.
\begin{theorem}\label{thm: convergence stochastic gradient}
Let $\varepsilon, \alpha_T, \beta_T \ge 0$, $0 \le \theta < 1$ and $\kappa < \theta/(1+ \theta)$. Consider a sequences $\{z^t\}_{t \ge 0}, \{v^t\}_{t \ge 0}$ generated by \eqref{eq: stochastic gradient descent} with arbitrary starting points $z^0, v^0 \in \bb C^d$ and define 
\begin{align}
\mu_t^{\max} & := \min \left\{ (1+t)^{-1 + \kappa} B^{-\frac{1}{1 - \theta}}(z^t,v^t), B_z^{-\frac{2}{3 - \theta}} (z^t,v^t), B_v^{-\frac{2}{3 -\theta}} (z^t,v^t), (1 - \tfrac{1}{K})^{-1/\theta} \right\} \label{eq: mu max}.
\end{align}
\begin{enumerate}
\item If the step sizes $\mu_t$ and $\nu_t$ are adapted to the filtration $\{ \q F_t \}_{t \ge 0}$ and satisfy $\mu_t, \nu_t \le \mu_t^{\max}$ then, the sequence $\q J(z^{t},v^{t})$ converges a.s. 
\item If $ \alpha_T, \beta_T, \theta > 0$, $\kappa \ge 0$ and the step sizes are given by $\mu_t = \mu \cdot \mu_t^{\max}$ and $\nu_t = \nu \cdot \mu_t^{\max}$ for some $0 < \mu,\nu \le 1$, then 
\[
\min_{t=0,\ldots, T -1} \norm{\nabla \q J(z^t,v^t)}_2^2 
\le
\begin{cases}
\frac{\kappa C_2}{\min\{\mu,\nu\}  [(1+T)^\kappa - 1] }, & \kappa > 0, \\
\frac{C_2}{\min\{\mu,\nu\} \ln (1+T) }, & \kappa = 0,
\end{cases}
\quad \text{a.s.,}
\]
for an a.s. finite random variable $C_2$ and 
\[
\inf_{t \ge 0} \norm{\nabla \q J(z^t,v^t)}_2^2  = 0 \quad \text{a.s.}
\]
\item If, in addition to assumptions in 2, $\varepsilon >0$, then $\norm{\nabla \q J(z^t,v^t)}_2 \to 0$ as $t \to \infty$ a.s. 
\end{enumerate}
\end{theorem}

Let us elaborate on the statement of \Cref{thm: convergence stochastic gradient}. In comparison to \Cref{thm: convergence gradient} two additional parameters $\kappa$ and $\theta$ appear. The first parameter $\kappa$ controls the decay of the step sizes over time and the speed at which the minimal gradient converges to zero. This is similar to the stochastic gradient descent convergence properties observed for PIE in nonblind ptychography \cite{Melnyk.2022b} and stochastic gradient in general \cite{Liu.2022}. The second, $\theta$, controls the power of the partial gradient normalization by the step sizes. In view of bounds in \Cref{l: stochastic gradient properties 1}, the second and the third cases in \eqref{eq: mu max} correspond to the norms of stochastic gradients, $\norm{g_z(z^t,v^t)}_2^{-2/(3 -\theta)}$ and $\norm{g_v(z^t,v^t)}_2^{-2/(3 -\theta)}$. If $\theta$ is set to zero, the power $-2/3$ would precisely correspond to the same powers in \Cref{thm: convergence gradient}. However, for stronger convergence guarantees 2 and 3 of \Cref{thm: convergence stochastic gradient}, we require a slightly stronger scaling $\theta > 0$. In the literature \cite{Khaled.2023, Liu.2022, Melnyk.2022b}, $\theta$ is commonly chosen as one. Yet, in our case, this leads to small updates in \eqref{eq: stochastic gradient descent},
\[
\norm{z^{t+1} - z^t}_2 = \mu_t \norm{g_z(z^t,v^t)}_2  \le 1, 
\quad 
\norm{v^{t+1} - v^t}_2 = \nu_t \norm{g_v(z^t,v^t)}_2  \le 1,
\] 
and motivates the restriction $0 \le \theta < 1$.

Similarly to the nonblind case \cite{Melnyk.2022b}, we observe that the convergence speed of the stochastic gradient descent \eqref{eq: stochastic gradient descent} is much slower than the gradient descent \eqref{eq: gradient descent}. This is compensated by a much larger computational complexity for the latter as the full gradient has to be evaluated. We also observe that smoothing $\varepsilon>0$ is necessary to guarantee that the gradient vanishes.

Coming back to ePIE, we showed in \eqref{eq: ePIE as stochastic gradient} that it can be seen as stochastic gradient descent with steps sizes $\mu_t = \alpha_T p_{r^t}/d \norm{v^t}_\infty^2$ and $\nu_t = \alpha_T p_{r^t}/d \norm{z^t}_\infty^2$. Note that they depend on $r^t$ unless $p$ is a uniform distribution over $\q R$. In such a case, we are able to apply \Cref{thm: convergence stochastic gradient} by choosing parameters $\alpha_t$ and $\beta_t$ appropriately. This would only provide the weakest convergence that the loss function eventually stops at a certain level. Hence, a stronger convergence properties can be obtained for ePIE by including Tickhonov regularization and smoothing $\varepsilon > 0$.

\subsection{Existence of larger step sizes}\label{sec: larger step size}

The main criticism of our results above is the choice of step sizes, which have to be sufficiently small to control the right-hand side in \Cref{col: descent lemma J}. However, the bound in \Cref{col: descent lemma J} significantly simplifies if one of the variables remains fixed.
\begin{theorem}[Version of \cite{Filbir.2022}] \label{thm: previous results}
Let $\varepsilon,\alpha_T, \beta_T \ge 0$. Consider $z_+ = z - L_v^{-1} \nabla_z \q J(z,v)$ and $v_+ = v - L_z^{-1} \nabla_v \q J(z,v)$ with $L_v$ and $L_z$ given by
\begin{align}
L_v & = d \max_{j = 1, \ldots, d} \sum_{r \in R} |(S_r v)_j|^2 + \alpha_T \le d \norm{v}_2^2 + \alpha_T, \label{eq: Lipschitz constants} \\
L_z & = d \max_{j = 1, \ldots, d} \sum_{r \in R} |(S_{-r} z)_j|^2 + \beta_T \le d \norm{z}_2^2 + \beta_T. \nonumber
\end{align}
Then,
\[
\q J(z_+, v) \le \q J(z, v) - L_v^{-1} \norm{\nabla_z \q J(z,v)}_2^2 
\quad \text{and} \quad
\q J(z, v_+) \le \q J(z, v) - L_z^{-1} \norm{\nabla_v \q J(z,v)}_2^2.
\]
%defined in \eqref{eq: Lipschitz constants}. 
\end{theorem} 
The step sizes $L_v^{-1}$ and $L_z^{-1}$ are much larger than their counterparts in \eqref{eq: step size gradient}. Therefore, alternating minimization procedures for blind ptychography such as in \cite{Hesse.2015, Chang.2019, Filbir.2022, Melnyk.2023} are more efficient than joint optimization. This naturally leads to a question if the step sizes for joint optimization \eqref{eq: gradient descent} can be chosen larger, preferably as in \Cref{thm: previous results}.   
Since $\q J$ is continuous, it attains its minimum on the closed interval connecting $(z_+,v)$ and $(z,v_+)$ consisting of the points $\{ ( z(\gamma), v(\gamma) ) : \gamma \in [0,1] \}$  given by
\begin{align}
z(\gamma) &=  z + \gamma (z_+ - z) =  z - \gamma L_v^{-1} \nabla_z \q J(z,v),
\ \label{eq: line interval} \\
v(\gamma) &= v_+ + \gamma (v - v_+) = v - (1-\gamma) L_z^{-1} \nabla_v \q J(z,v). \nonumber
\end{align}
Hence, there exists $\gamma_* \in [0,1]$ that for a point $(z(\gamma_*), v(\gamma_*) )$ the values $\q J(z(\gamma_*), v(\gamma_*)) \le \min \{\q J(z_+,v), \q J(z,v_+)\}$ and, consequently, both inequalities in \Cref{thm: previous results} hold. If the iterates are selected this way, we obtain the following convergence guarantees. 
\begin{theorem}\label{thm: smart step size}
Let $\varepsilon \ge 0, \alpha_T, \beta_T > 0$. Consider the sequences $\{ z^t\}_{t \ge 0}$ and $\{v^t\}_{t \ge 0}$ to be the sequences where pair $z^{t+1}, v^{t+1}$ is constructed via \eqref{eq: line interval} with $z = z^{t}$, $v = v^{t}$ and $\gamma$ chosen such that 
\begin{equation} \label{eq: gamma choice}
\q J(z^{t+1}, v^{t+1}) 
= \q J(z(\gamma), v(\gamma)) \le \min\{\q J(z_+,v), \q J(z, v_+) \}.
\end{equation}
Then, we have
\begin{align*} 
& \q J(z^{t+1}, v^{t+1}) 
\le \q J(z^t, v^t) 
- \tfrac{1}{2} L_{z^t}^{-1} \norm{ \nabla_z \q J(z^t,v^t) }_2^2 
- \tfrac{1}{2} L_{v^t}^{-1} \norm{ \nabla_v \q J(z^t,v^t)}_2^2, \quad t \ge 0,
\end{align*}
$\norm{ \nabla \q J(z^t,v^t) }_2^2  \to 0 \text{ as } t \to \infty$ and
\begin{align*}
\min_{t=0,\ldots, T-1} \norm{ \nabla \q J(z^t,v^t) }_2^2 
\le T^{-1} \left[ d \max\{ \alpha_T^{-1}, \beta_T^{-1} \} \q J(z^0, v^0) + \max\{ \alpha_T, \beta_T \} \right] \cdot \left[ \q J(z^0, v^0) -  \q J^{\inf}\right].
\end{align*}
\end{theorem}
\begin{proof}
The first inequality follows from \eqref{eq: gamma choice} and \Cref{thm: previous results}. The rest of the proof is similar to the proof of \Cref{thm: convergence gradient} or Theorem 3.4 in \cite{Filbir.2022}.
\end{proof}
\Cref{thm: smart step size} is similar to Theorem 3.4 in \cite{Filbir.2022}, but it allows for simultaneous change of $z^t$ and $v^t$ unlike the alternating minimization procedure in in \cite{Filbir.2022}. However, it comes with extra computational costs resulting from a search for $\gamma$ at each iteration. The minimal increase is two times as $\gamma$ can be chosen as $\gamma = 1$ if $\q J(z_+, v) > \q J(z,v_+)$ and $\gamma = 0$ otherwise, which requires two evaluations of the loss $\q J$ instead of single evaluation in \cite{Filbir.2022}. An additional outcome is that the continuity of the gradient is not necessary, i.e., $\varepsilon$ can be zero, unlike in \cite{Filbir.2022}. 

This idea could potentially be used for stochastic gradient descent. However, in this case, the evaluation of the function value is costly and is almost the same as the computation of the full gradient, which negates the computational speed-up of the stochastic gradient. 

\section{Proofs}\label{sec: proofs}

\subsection{Wirtinger derivatives and derivation of descent lemmas}\label{sec: Wirtinger derivatives}

The proof of descent lemma consists of several steps. Firstly, we rewrite $\q L_\varepsilon$ from \eqref{eq: loss} in a more general form
\begin{equation}\label{eq: loss general}
\q L_\varepsilon(z,v) = \sum_{j = 1}^m \left[ \sqrt{ |z^T Q_j v|^2 + \varepsilon} - \sqrt{y_j + \varepsilon} \right]^2.
\end{equation}
In order to do so, let us first define vectors $f^k$ with $f^k_\ell = e^{2 \pi i k \ell / d}$, $\ell \in [d]$. Then, we can rewrite
\begin{align}
[ F (z \circ S_r v) ]_k 
& = (f^k)^* (z \circ S_r v) 
=  (f^k)^* \diag(z)  S_r v
= (\diag(z) \conj{f^k})^T S_r v \nonumber \\
& = (\diag(\conj{f^k}) z)^T S_r v
= z^T  \diag(\conj{f^k}) S_r v
=: z^T Q_{r,k} v. \label{eq: matrix Q definition}
\end{align}
Hence, with the indexing $j = (r,k)$, $r \in \q R, k \in [d]$ and $m:= d R$, we can instead work with \eqref{eq: loss general}. 

The analysis of $\q L_\varepsilon$ is based on the Wirtinger calculus, which becomes handy for the optimization of real-valued functions of complex variables. Note that the class of such functions is not holomorphic \cite[Proposition 4.0.1]{Hunger.2008}. Let us consider a differentiable (in the real sense) function $f: \bb C \to \bb R$ with an argument $z = \RE z + i \IM z$. Then, the Wirtinger derivatives are defined as
\[
\frac{\partial f}{\partial z} := \frac{1}{2}\frac{\partial f}{\partial \RE}  - \frac{i}{2}\frac{\partial f}{\partial \IM},
\quad 
\frac{\partial f}{\partial \conj{z} } := \frac{1}{2}\frac{\partial f}{\partial \RE}  + \frac{i}{2}\frac{\partial f}{\partial \IM}.
\] 
They are a linear transformation of the derivatives with respect to real and imaginary parts $\RE z$ and $\IM z$ and, hence, the standard results such as derivation rules, extension to the multivariate case and etc. hold for Wirtinger derivatives. In addition, the conjugation rule applies,
\begin{equation}\label{eq: conj rule}
\frac{\partial f}{\partial \conj z} = \conj{ \frac{\partial  f}{\partial z} }.
\end{equation}   
Now, let us consider $f: \bb C^d  \to \bb R$ with Wirtinger derivatives
\[
\frac{\partial f}{\partial z} = \left( \frac{\partial f}{\partial z_1}, \ldots,  \frac{\partial f}{\partial z_d} \right)
\quad
\text{and}
\quad
\frac{\partial f}{\partial \conj{z} } = \left( \frac{\partial f}{\partial \conj z_1}, \ldots,  \frac{\partial f}{\partial \conj z_d} \right).
\]
The differential of $f$ is given by
\[
d f = \sum_{j =1}^d\left[ \frac{\partial f}{\partial \alpha_j} d \alpha_j + \frac{\partial f}{\partial \beta_j} d \beta_j \right]
= \sum_{j =1}^d \left[ \frac{\partial f}{\partial z_j} d z_j + \frac{\partial f}{\partial \conj z_j} d \conj z_j \right]
= \frac{\partial f}{\partial z} d z + \frac{\partial f}{\partial \conj z} d \conj z 
= 2 \RE\left[ \frac{\partial f}{\partial z} d z \right],
\]
where \eqref{eq: conj rule} and real-valuedness of $f$ were used in the last step. Consequently, the direction of the steepest ascend $d z$ is aligned with the conjugate transpose of $\frac{\partial f}{\partial z}$, which is why Wirtinger gradients in direction of $z$ and $\conj z$ are defined as
\[
\nabla_z f := \left(\frac{\partial f}{\partial z} \right)^* = \left(\frac{\partial f}{\partial \conj z} \right)^T
\text{ and }
\nabla_{\conj z} f :=  \conj{\nabla_z f}.
\]
Furthermore, for twice differentiable $f$ we would need the second-order Taylor expansion
\begin{equation}\label{eq: Taylor}
f(z + s) = f(z) + 2 \RE(s^* \nabla_z f(z)) + \begin{bmatrix} s \\ \conj s \end{bmatrix}^*
\int_0^1 (1 - t) \nabla^2 f(z + t s) d t \begin{bmatrix} s \\ \conj s \end{bmatrix},
\end{equation}
where $z,s \in \bb C^d$ and $\nabla^2 f$ is the Hessian matrix
\[
\nabla^2 f = 
{
\renewcommand\arraystretch{1.75}
\begin{bmatrix}
\nabla_{z,z}^2 f & \nabla_{\conj z, z}^2 f \\
\nabla_{z, \conj{z}}^2 f & \nabla_{\conj{z}, \conj{z}}^2 f
\end{bmatrix}
}
=
{
\renewcommand\arraystretch{1.75}
\begin{bmatrix}
\frac{\partial}{\partial z}  \nabla_z f & \frac{\partial}{\partial \conj z}  \nabla_z f \\
\frac{\partial}{\partial z}  \nabla_{\conj{z}} f & \frac{\partial}{\partial \conj{z}}  \nabla_{\conj{z}} f
\end{bmatrix}
}
=
{
\renewcommand\arraystretch{1.75}
\begin{bmatrix}
\frac{\partial}{\partial z}  \nabla_z f & \frac{\partial}{\partial \conj z}  \nabla_z f \\
\conj{ \frac{\partial}{\partial \conj{z}}  \nabla_{z} f } & \conj{ \frac{\partial}{\partial z}  \nabla_{z} f }
\end{bmatrix},
}
\]
where the last equality follows from \eqref{eq: conj rule}. Consequently, the second-order term in \eqref{eq: Taylor} simplifies to
\begin{equation}\label{eq: second order simple}
\begin{bmatrix} s \\ \conj s \end{bmatrix}^*
\nabla^2 f(z)
\begin{bmatrix} s \\ \conj s \end{bmatrix}
= 2 \RE \left ( s^* \nabla_{z,z}^2 f(z) s + s^*  \nabla_{\conj z, z}^2 f(z) \conj s \right).
\end{equation}
The descent lemma is generally established by bounding from above the second-order term in \eqref{eq: Taylor} and \eqref{eq: second order simple} is a convenient representation to work with. 

Returning to $\q L_\varepsilon$ and $\q J$, we start with the case $\varepsilon >0$ so that $\q L_\varepsilon$ is twice continuously differentiable. Since the argument of $\q L_\varepsilon$ is two variables $z$ and $v$ we will use notation 
\[
\nabla \q L_\varepsilon 
:= \nabla_{(z,v)} \q L_\varepsilon
= \begin{bmatrix} \nabla_z \q L_\varepsilon \\ \nabla_v \q L_\varepsilon \end{bmatrix}
\quad \text{and} \quad 
\nabla_{\conj{(z,v)}} \q L_\varepsilon = \conj{\nabla \q L_\varepsilon}. 
\]
In the case $\varepsilon = 0$, $\q L_0$ is not differentiable at points where $z^T Q_j v = 0$ for at least one $j \in [m]$. Thus, for $\q L_0$ we define the generalized Wirtinger gradient as pointwise limit
\[
\nabla \q L_0 (z,v) = \lim_{\varepsilon \to 0+} \nabla \q L_\varepsilon (z,v). 
\]
Following this logic, for the case $\varepsilon > 0$ we split the blocks of the Hessian matrix into subblocks as
\[
\nabla^2_{(z,v),(z,v)} \q L_\varepsilon
={
\renewcommand\arraystretch{1.75}
\begin{bmatrix}
\nabla_{z,z}^2 \q L_\varepsilon & \nabla_{v,z}^2 \q L_\varepsilon \\
\nabla_{z, v}^2 \q L_\varepsilon & \nabla_{v, v}^2 \q L_\varepsilon
\end{bmatrix}
}
\quad \text{and} \quad
\nabla^2_{\conj{(z,v)},(z,v)} \q L_\varepsilon
={
\renewcommand\arraystretch{1.75}
\begin{bmatrix}
\nabla_{\conj{z},z}^2 \q L_\varepsilon & \nabla_{\conj{v},z}^2 \q L_\varepsilon \\
\nabla_{\conj{z}, v}^2 \q L_\varepsilon & \nabla_{\conj{v}, v}^2 \q L_\varepsilon
\end{bmatrix}.
}
\]
Consequently, the vector $s$ in \eqref{eq: second order simple} splits into two parts, $s := (u,h)$ corresponding to change $(z + u, v + h)$. Hence, \eqref{eq: second order simple} becomes
\begin{equation}\label{eq: quadratic simple L}
\begin{bmatrix} s \\ \conj s \end{bmatrix}^*
\nabla^2 \q L_\varepsilon(z,v)
\begin{bmatrix} s \\ \conj s \end{bmatrix}
= 2 \RE \left ( \begin{bmatrix} u \\ h \end{bmatrix}^* \nabla^2_{(z,v),(z,v)} \q L_\varepsilon(z,v) \begin{bmatrix} u \\ h \end{bmatrix} 
+ \begin{bmatrix} u \\ h \end{bmatrix}^* \nabla^2_{\conj{(z,v)},(z,v)} \q L_\varepsilon(z,v) \begin{bmatrix} \conj u \\ \conj h \end{bmatrix}\right).
\end{equation}

The next lemma summarizes the computation for both summands on the right-hand side and provides a suitable lower and upper bounds.
\begin{lemma}\label{l: Hessian bound}
For all $s = (u,h) \in \bb C^{2d}$ we have
%\begin{align*}
%\begin{bmatrix} u \\ h \end{bmatrix}^* \nabla^2_{(z,v),(z,v)} \q L_\varepsilon(z,v) \begin{bmatrix} u \\ h \end{bmatrix}
%& = \sum_{j=1}^m \left[1 - \frac{\varepsilon \sqrt{y_j + \varepsilon}}{(|z^T Q_j v|^2 + \varepsilon)^{3/2} } - \frac{\sqrt{y_j + \varepsilon} |z^T Q_j v|^2}{2(|z^T Q_j v|^2 + \varepsilon)^{3/2}} \right] |u^T Q_j v + z^T Q_j h |^2, \\
%\RE \left( \begin{bmatrix} u \\ h \end{bmatrix}^* \nabla^2_{\conj{(z,v)},(z,v)} \q L_\varepsilon(z,v) \begin{bmatrix} \conj u \\ \conj h \end{bmatrix}\right)
%=  
%\end{align*}
%Furthermore, for $s = (u,h)$ we have
\[
\begin{bmatrix} s \\ \conj s \end{bmatrix}^*
\nabla^2 \q L_\varepsilon(z,v)
\begin{bmatrix} s \\ \conj s \end{bmatrix}
\le 2 \sum_{j=1}^m  |u^T Q_j v + z^T Q_j h |^2 + 4 \left[ \q L_\varepsilon(z,v) \right]^{1/2} \cdot \left[ \sum_{j=1}^m |u^T Q_j h|^2 \right]^{1/2}.
\]
and
\begin{align*}
\begin{bmatrix} s \\ \conj s \end{bmatrix}^*
\nabla^2 \q L_\varepsilon(z,v)
\begin{bmatrix} s \\ \conj s \end{bmatrix}
& \ge  - 4\left[ \q L_\varepsilon(z,v) \right]^{1/2} \cdot \left[ \sum_{j=1}^m |u^T Q_j h|^2 \right]^{1/2} \\
& \quad  - 2 (\norm{y + \varepsilon}_{\infty}^{1/2} \varepsilon^{-1/2} - 1) \sum_{j=1}^m |u^T Q_j v + z^T Q_j h |^2. 
\end{align*}
\end{lemma}
\begin{proof}
In order to reduce technical computations, we make use of the results for phase retrieval \cite{Xu.2018}. That is, let us consider supplementary functions 
\begin{equation*}\label{eq: function fa}
f_a(z) := \sum_{j = 1}^m \left[ \sqrt{ |a_j^*z|^2 + \varepsilon} - \sqrt{y_j + \varepsilon} \right]^2 
\quad \text{and} \quad 
f_b(v) := \sum_{j = 1}^m \left[ \sqrt{ |b_j^* v|^2 + \varepsilon} - \sqrt{y_j + \varepsilon} \right]^2.
\end{equation*}
Note that $\q L_\varepsilon(z,v) = f_a(z)$ with $a_j = \conj{Q_j v}$ and $\q L_{\varepsilon}(z,v) = f_b(v)$ with $b_j = \conj{Q_j^T z}$. Hence, we can use the computations in \cite[pp.26-27]{Xu.2018} to obtain
\begin{align*}
\nabla_z \q L_\varepsilon(z,v) & = \nabla_z f_a(z) = \sum_{j=1}^m \left[1 - \frac{\sqrt{y_j + \varepsilon}}{\sqrt{|z^T Q_j v|^2 + \varepsilon}} \right] z^T Q_j v \cdot \conj{Q_j v}, \\
\nabla_v \q L_\varepsilon(z,v) & = \nabla_v f_b(v) = \sum_{j=1}^m \left[1 - \frac{\sqrt{y_j + \varepsilon}}{\sqrt{|z^T Q_j v|^2 + \varepsilon}} \right] z^T Q_j v \cdot \conj{Q_j^T z},
\end{align*}
and
\begin{align*}
\nabla_{z,z} \q L_\varepsilon(z,v) & = \nabla_{z,z} f_a(z) = \sum_{j=1}^m \left[1 - \frac{\varepsilon \sqrt{y_j + \varepsilon}}{(|z^T Q_j v|^2 + \varepsilon)^{3/2}} - \frac{ \sqrt{y_j + \varepsilon} |z^T Q_j v|^2 }{2(|z^T Q_j v|^2 + \varepsilon)^{3/2}} \right] \conj{Q_j v} (Q_j v)^T, \\
\nabla_{\conj z,z} \q L_\varepsilon(z,v) & = \nabla_{\conj z,z} f_a(z) = \sum_{j=1}^m \frac{\sqrt{y_j + \varepsilon} (z^T Q_j v)^2}{2(|z^T Q_j v|^2 + \varepsilon)^{3/2}}  \conj{Q_j v} (\conj{Q_j v})^T, \\
\nabla_{v,v} \q L_\varepsilon(z,v) & = \nabla_{v,v} f_b(v) = \sum_{j=1}^m \left[1 - \frac{\varepsilon \sqrt{y_j + \varepsilon}}{(|z^T Q_j v|^2 + \varepsilon)^{3/2}} - \frac{ \sqrt{y_j + \varepsilon} |z^T Q_j v|^2 }{2(|z^T Q_j v|^2 + \varepsilon)^{3/2}} \right] \conj{Q_j^T z} (Q_j^T z)^T, \\
\nabla_{\conj v,v} \q L_\varepsilon(z,v) & = \nabla_{\conj v,v} f_b(v) = \sum_{j=1}^m \frac{\sqrt{y_j + \varepsilon} (z^T Q_j v)^2}{2(|z^T Q_j v|^2 + \varepsilon)^{3/2}}  \conj{Q_j^T z} (\conj{Q_j^T z})^T.
\end{align*}
Now, we have to compute the cross-variable derivatives using the product and the chain rules. More precisely,
\begin{align*}
\nabla_{v,z} \q L_\varepsilon(z,v) 
& = \frac{\partial }{\partial v} \nabla_{z} \q L_\varepsilon(z,v) 
= \sum_{j=1}^m \left[1 - \frac{\sqrt{y_j + \varepsilon}}{\sqrt{|z^T Q_j v|^2 + \varepsilon} } \right] \conj{Q_j v} \cdot \frac{\partial (z^T Q_j v)}{\partial v} \\
& \quad \quad  - \sum_{j=1}^m \sqrt{y_j + \varepsilon} (z^T Q_j v) \conj{Q_j v} \cdot \frac{\partial (|z^T Q_j v|^2 + \varepsilon)^{-1/2} }{\partial |z^T Q_j v|^2 + \varepsilon} \frac{\partial (\conj{z^T Q_j v}  z^T Q_j v)}{\partial v} \\
& = \sum_{j=1}^m \left[1 - \frac{\sqrt{y_j + \varepsilon}}{\sqrt{|z^T Q_j v|^2 + \varepsilon} } + \frac{\sqrt{y_j + \varepsilon} |z^T Q_j v|^2}{2(|z^T Q_j v|^2 + \varepsilon)^{3/2}} \right] \conj{Q_j v}  (Q_j^T z)^T \\
& = \sum_{j=1}^m \left[1 - \frac{\varepsilon \sqrt{y_j + \varepsilon}}{(|z^T Q_j v|^2 + \varepsilon)^{3/2} } - \frac{\sqrt{y_j + \varepsilon} |z^T Q_j v|^2}{2(|z^T Q_j v|^2 + \varepsilon)^{3/2}} \right] \conj{Q_j v}  (Q_j^T z)^T,
\end{align*}
and 
\begin{align*}
\nabla_{\conj{v}, z} \q L_\varepsilon(z,v) 
& = \sum_{j=1}^m \left[1 - \frac{\sqrt{y_j + \varepsilon}}{\sqrt{|z^T Q_j v|^2 + \varepsilon} } \right] z^T Q_j v   \cdot \frac{\partial(\conj{Q_j v})}{\partial \conj v} \\
& \quad \quad  - \sum_{j=1}^m \sqrt{y_j + \varepsilon} (z^T Q_j v) \conj{Q_j v} \cdot \frac{\partial (|z^T Q_j v|^2 + \varepsilon)^{-1/2} }{\partial |z^T Q_j v|^2 + \varepsilon} \frac{ \partial( z^T Q_j v \conj{z^T Q_j v}) }{\partial \conj v} \\
& = \sum_{j=1}^m \left[1 - \frac{\sqrt{y_j + \varepsilon}}{\sqrt{|z^T Q_j v|^2 + \varepsilon} } \right] (z^T Q_j v) \,  \conj{Q_j}
+ \sum_{j=1}^m \frac{\sqrt{y_j + \varepsilon} (z^T Q_j v)^2 }{2(|z^T Q_j v|^2 + \varepsilon)^{3/2} } \conj{Q_j v} (\conj{Q_j^T z})^T.
\end{align*}
The last two derivatives are computed analogously,
\begin{align*}
\nabla_{z,v} \q L_\varepsilon(z,v) 
& = \sum_{j=1}^m \left[1 - \frac{\varepsilon \sqrt{y_j + \varepsilon}}{(|z^T Q_j v|^2 + \varepsilon)^{3/2} } - \frac{\sqrt{y_j + \varepsilon} |z^T Q_j v|^2}{2(|z^T Q_j v|^2 + \varepsilon)^{3/2}} \right] \conj{Q_j^T z}  (Q_j v)^T, \\
\nabla_{\conj{z},v} \q L_\varepsilon(z,v) 
& = \sum_{j=1}^m \left[1 - \frac{\sqrt{y_j + \varepsilon}}{\sqrt{|z^T Q_j v|^2 + \varepsilon} } \right] (z^T Q_j v) \, \conj{Q_j}^T
+ \sum_{j=1}^m \frac{\sqrt{y_j + \varepsilon} (z^T Q_j v)^2 }{2(|z^T Q_j v|^2 + \varepsilon)^{3/2} } \conj{Q_j^T z} (\conj{Q_j v})^T.
\end{align*}
With derivatives computed, we can now turn to the quadratic terms in \eqref{eq: quadratic simple L}. The first summand is given by
\begin{align}
& \begin{bmatrix} u \\ h \end{bmatrix}^* \nabla^2_{(z,v),(z,v)} \q L_\varepsilon(z,v) \begin{bmatrix} u \\ h \end{bmatrix} \nonumber  \\
& \quad \quad = u^* \nabla_{z,z} \q L_\varepsilon(z,v) u + u^* \nabla_{v,z} \q L_\varepsilon(z,v) h
+ h^* \nabla_{z,v} \q L_\varepsilon(z,v) u + h^* \nabla_{v,v} \q L_\varepsilon(z,v) h \nonumber \\
& \quad \quad = \sum_{j=1}^m \left[1 - \frac{\varepsilon \sqrt{y_j + \varepsilon}}{(|z^T Q_j v|^2 + \varepsilon)^{3/2} } - \frac{\sqrt{y_j + \varepsilon} |z^T Q_j v|^2}{2(|z^T Q_j v|^2 + \varepsilon)^{3/2}} \right] \times \nonumber \\
& \quad \quad \quad \quad \quad \quad \times \left[ |u^T Q_j v|^2 + \conj{u^T Q_j v} \cdot z^T Q_j h +  \conj{ z^T Q_j h} \cdot u^T Q_j v + |z^T Q_j h|^2 \right] \nonumber \\
& \quad \quad = \sum_{j=1}^m \left[1 - \frac{\varepsilon \sqrt{y_j + \varepsilon}}{(|z^T Q_j v|^2 + \varepsilon)^{3/2} } - \frac{\sqrt{y_j + \varepsilon} |z^T Q_j v|^2}{2(|z^T Q_j v|^2 + \varepsilon)^{3/2}} \right] |u^T Q_j v + z^T Q_j h |^2. \label{eq: Hessian top left}
\end{align}
For the second summand, we analogously obtain
\begin{align}
& \RE \left( \begin{bmatrix} u \\ h \end{bmatrix}^* \nabla^2_{\conj{(z,v)},(z,v)} \q L_\varepsilon(z,v) \begin{bmatrix} \conj u \\ \conj h \end{bmatrix} \right) \nonumber \\
& \quad \quad  = \RE( u^* \nabla_{\conj z,z} \q L_\varepsilon(z,v) \conj u + u^* \nabla_{\conj v,z} \q L_\varepsilon(z,v) \conj h
+ h^* \nabla_{\conj z,v} \q L_\varepsilon(z,v) \conj u + h^* \nabla_{\conj v,v} \q L_\varepsilon(z,v) \conj h ) \nonumber \\
&  \quad \quad  =  2 \sum_{j=1}^m \left[1 - \frac{\sqrt{y_j + \varepsilon}}{\sqrt{|z^T Q_j v|^2 + \varepsilon} } \right] \RE( z^T Q_j v \cdot \conj{u^T Q_j h} ) \nonumber \\
& \quad \quad \quad \quad + \RE \left( \sum_{j=1}^m \frac{\sqrt{y_j + \varepsilon} (z^T Q_j v)^2 }{2(|z^T Q_j v|^2 + \varepsilon)^{3/2} } \left[ (\conj{u^T Q_j v})^2  + 2 \conj{u^T Q_j v \cdot z^T Q_j h} + (\conj{z^T Q_j h})^2 \right] \right) \nonumber \\
& \quad \quad  =  2 \sum_{j=1}^m \left[1 - \frac{\sqrt{y_j + \varepsilon}}{\sqrt{|z^T Q_j v|^2 + \varepsilon} } \right] \RE( z^T Q_j v \cdot \conj{u^T Q_j h} ) \label{eq: Hessian top right}\\
& \quad \quad \quad \quad + \sum_{j=1}^m \frac{\sqrt{y_j + \varepsilon}  }{2(|z^T Q_j v|^2 + \varepsilon)^{3/2} } \RE \left(  (z^T Q_j v)^2 (\conj{u^T Q_j v + z^T Q_j h} )^2 \right)  \nonumber
\end{align}
To complete the upper bound in \eqref{eq: quadratic simple L}, we make the following observations. Firstly, the second summands in \eqref{eq: Hessian top left} are nonpositive, 
\[
- \sum_{j=1}^m \frac{\varepsilon \sqrt{y_j + \varepsilon}}{(|z^T Q_j v|^2 + \varepsilon)^{3/2} } |u^T Q_j v + z^T Q_j h |^2 \le 0.
\]
Secondly, the third summands in \eqref{eq: Hessian top left} and the second sum in \eqref{eq: Hessian top right} added together admit
\[
\RE \left(  (z^T Q_j v)^2 (\conj{u^T Q_j v + z^T Q_j h} )^2 \right) - |z^T Q_j v|^2 |\conj{u^T Q_j v + z^T Q_j h} |^2 
\le 0.
\]
Since $\sqrt{y_j + \varepsilon} /(|z^T Q_j v|^2 + \varepsilon)^{3/2} > 0$, the above inequality applies for the whole sum. Thirdly, we bound the first sum in \eqref{eq: Hessian top right} by the Cauchy-Schwarz inequality as
\begin{align}
& \left| \sum_{j=1}^m \left[1 - \frac{\sqrt{y_j + \varepsilon}}{\sqrt{|z^T Q_j v|^2 + \varepsilon} } \right] \RE( z^T Q_j v \cdot \conj{u^T Q_j h} ) \right| \nonumber \\
& \quad \quad \quad \quad \le \sum_{j=1}^m \left| 1 - \frac{\sqrt{y_j + \varepsilon}}{\sqrt{|z^T Q_j v|^2 + \varepsilon} } \right| |z^T Q_j v| |u^T Q_j h| \nonumber \\
& \quad \quad \quad \quad \le \left[ \sum_{j=1}^m \left| \sqrt{|z^T Q_j v|^2 + \varepsilon} - \sqrt{y_j + \varepsilon}\right|^2 \frac{|z^T Q_j v|^2}{|z^T Q_j v|^2 + \varepsilon} \right]^{1/2} \cdot \left[ \sum_{j=1}^m |u^T Q_j h|^2 \right]^{1/2} \nonumber \\
& \quad \quad \quad \quad \le \left[ \q L_\varepsilon(z,v) \right]^{1/2} \cdot \left[ \sum_{j=1}^m |u^T Q_j h|^2 \right]^{1/2}. \label{eq: gradient bound later}
\end{align}
Combining these inequalities with \eqref{eq: quadratic simple L} yields
\[
\begin{bmatrix} s \\ \conj s \end{bmatrix}^*
\nabla^2 \q L_\varepsilon(z,v)
\begin{bmatrix} s \\ \conj s \end{bmatrix}
\le 2 \sum_{j=1}^m  |u^T Q_j v + z^T Q_j h |^2 + 4 \left[ \q L_\varepsilon(z,v) \right]^{1/2} \cdot \left[ \sum_{j=1}^m |u^T Q_j h|^2 \right]^{1/2}.
\]
For the lower bound, we return to \eqref{eq: Hessian top right}. The absolute value of the first summand was already bounded in \eqref{eq: gradient bound later}. The second summand can be bounded using
\[
\RE \left(  (z^T Q_j v)^2 (\conj{u^T Q_j v + z^T Q_j h} )^2 \right)
\ge - |z^T Q_j v|^2 |u^T Q_j v + z^T Q_j h|^2.
\]
This gives
\begin{align*}
\RE \left( \begin{bmatrix} u \\ h \end{bmatrix}^* \nabla^2_{\conj{(z,v)},(z,v)} \q L_\varepsilon(z,v) \begin{bmatrix} \conj u \\ \conj h \end{bmatrix} \right)  
& \ge - 2 \left[ \q L_\varepsilon(z,v) \right]^{1/2} \cdot \left[ \sum_{j=1}^m |u^T Q_j h|^2 \right]^{1/2} \\
& \quad - \sum_{j=1}^m \frac{\sqrt{y_j + \varepsilon} |z^T Q_j v|^2  }{2(|z^T Q_j v|^2 + \varepsilon)^{3/2} }  |u^T Q_j v + z^T Q_j h|^2.
\end{align*}
Thus, by \eqref{eq: quadratic simple L} and \eqref{eq: Hessian top left}, we obtain
\begin{align*}
& \begin{bmatrix} s \\ \conj s \end{bmatrix}^*
\nabla^2 \q L_\varepsilon(z,v)
\begin{bmatrix} s \\ \conj s \end{bmatrix}
\ge  - 4\left[ \q L_\varepsilon(z,v) \right]^{1/2} \cdot \left[ \sum_{j=1}^m |u^T Q_j h|^2 \right]^{1/2} \\
& \quad \quad \quad \quad + 2 \sum_{j=1}^m \left[1 - \frac{\varepsilon \sqrt{y_j + \varepsilon}}{(|z^T Q_j v|^2 + \varepsilon)^{3/2} } - \frac{\sqrt{y_j + \varepsilon} |z^T Q_j v|^2}{(|z^T Q_j v|^2 + \varepsilon)^{3/2}} \right] |u^T Q_j v + z^T Q_j h |^2 
\end{align*}
The last step is to note that
\begin{align*}
& 1 - \frac{\varepsilon \sqrt{y_j + \varepsilon}}{(|z^T Q_j v|^2 + \varepsilon)^{3/2} } - \frac{\sqrt{y_j + \varepsilon} |z^T Q_j v|^2}{(|z^T Q_j v|^2 + \varepsilon)^{3/2}}
= 1 - \frac{\sqrt{y_j + \varepsilon}}{( |z^T Q_j v|^2 + \varepsilon)^{1/2} } 
\ge 1 - \frac{\sqrt{y_j + \varepsilon}}{\sqrt{\varepsilon} }.
\end{align*}
The term $\varepsilon^{-1/2} (y_j + \varepsilon)^{1/2} - 1$ is nonnegative, which gives 
\begin{align*}
\begin{bmatrix} s \\ \conj s \end{bmatrix}^*
\nabla^2 \q L_\varepsilon(z,v)
\begin{bmatrix} s \\ \conj s \end{bmatrix}
& \ge  - 4\left[ \q L_\varepsilon(z,v) \right]^{1/2} \cdot \left[ \sum_{j=1}^m |u^T Q_j h|^2 \right]^{1/2} \\
& \quad  - 2 (\norm{y + \varepsilon}_{\infty}^{1/2} \varepsilon^{-1/2} - 1) \sum_{j=1}^m |u^T Q_j v + z^T Q_j h |^2. 
\end{align*}
\end{proof}

With \Cref{l: Hessian bound} we are able to derive the descent lemma for $\q J$.

\begin{proof}[Proof of \Cref{col: descent lemma J}.]
Let us start by deriving a descent lemma for $\q L_\varepsilon$ with $\varepsilon > 0$ first. Then, $\q L_\varepsilon$ is twice differentiable and the result can be derived from the second-order Taylor expansion \eqref{eq: Taylor} combined with the bound on the Hessian provided by \Cref{l: Hessian bound}. More precisely, with $s=(u,h)$
\begin{align*}
\q L_\varepsilon(z+u, v+h) 
& = \q L_\varepsilon(z, v) 
+ 2 \RE( u^* \nabla_z \q L_\varepsilon(z,v) ) 
+ 2 \RE(h^*\nabla_v \q L_\varepsilon(z,v)) \\
& \quad + \begin{bmatrix} s \\ \conj s \end{bmatrix}^*
\int_0^1 (1 - t) \nabla^2 \q L_\varepsilon(z + t u, v +th) d t \begin{bmatrix} s \\ \conj s \end{bmatrix}.
\end{align*}
Let us focus on the second-order term for now. By \Cref{l: Hessian bound}, we get
\begin{align}
\begin{bmatrix} s \\ \conj s \end{bmatrix}^*
& \int_0^1 (1 - t) \nabla^2 \q L_\varepsilon(z + t u, v +th) d t \begin{bmatrix} s \\ \conj s \end{bmatrix} \\
& \quad \quad \quad \quad \le 2 \int_0^1 (1 - t) \sum_{j=1}^m  |u^T Q_j (v + th) + (z+tu)^T Q_j h |^2 dt \nonumber \\
& \quad \quad \quad \quad + 4 \int_0^1 (1 - t)\left[ \q L_\varepsilon(z+tu,v+th) \right]^{1/2} \cdot \left[ \sum_{j=1}^m |u^T Q_j h|^2 \right]^{1/2} dt. \label{eq: descent lemma tech 1}
\end{align}
The next step is to simplify integrands. By \eqref{eq: matrix Q definition} and properties of the discrete Fourier Transform, for any $z,v \in \bb C^d$ we have
\begin{align}
& \sum_{j=1}^m  |z^T Q_j v|^2 
= \sum_{r \in \q R} \sum_{k = 1}^d |F(z \circ S_r v)_k|^2
= \sum_{r \in \q R} \norm{F(z \circ S_r v)}_2^2 \nonumber \\
& \quad = d \sum_{r \in \q R} \norm{z \circ S_r v}_2^2
= d \sum_{r \in \q R} \sum_{k = 1}^d |z_k|^2 |v_{k-r}|^2
= d \sum_{k = 1}^d |z_k|^2 \sum_{r \in \q R}  |v_{k-r}|^2
\le d  \norm{z}_2^2 \norm{v}_2^2. \label{eq: bilinear norm bound}
\end{align}
Consequently, using $|\alpha + \beta|^2 \le 2|\alpha|^2 + 2 |\beta|^2$, the first integral in \eqref{eq: descent lemma tech 1} simplifies to
\begin{align*}
& 2\int_0^1 (1 - t) \sum_{j=1}^m  |u^T Q_j (v + th) + (z+tu)^T Q_j h |^2 dt \\
& \quad \quad \quad \quad \le 4 d \norm{u}_2^2 \int_0^1 (1 - t) \norm{v + th}_2^2 dt 
+ 4 d \norm{h}_2^2 \int_0^1 (1 - t) \norm{z + tv}_2^2 dt.  
\end{align*}
Furthermore, both of the obtained integrals can be bounded as
\begin{align}
& \int_0^1 (1 - t) \norm{v + th}_2^2 dt 
 = \int_0^1 (1 - t) \norm{v}_2^2 + 2(1 - t) t\RE(v^*h) + (1-t) t^2 \norm{h}_2^2 dt \label{eq: descent lemma tech 2}\\
& = \tfrac{1}{2}\norm{v}_2^2 + \tfrac{1}{3}\RE(v^*h) + \tfrac{1}{12}\norm{h}_2^2 \le \tfrac{1}{2}\norm{v}_2^2 +\tfrac{1}{6} \norm{v}_2^2 + \tfrac{1}{6}\norm{h}_2^2 + \tfrac{1}{12}\norm{h}_2^2 
= \tfrac{2}{3}\norm{v}_2^2 + \tfrac{1}{4} \norm{h}_2^2, \nonumber
\end{align}
which gives 
\begin{align}
& 2 \int_0^1 (1 - t) \sum_{j=1}^m  |u^T Q_j (v + th) + (z+tu)^T Q_j h |^2 dt \nonumber \\
& \quad \quad \quad \quad \quad \quad \le 4 d \norm{u}_2^2 [\, \tfrac{2}{3}\norm{v}_2^2 + \tfrac{1}{4} \norm{h}_2^2 \,] 
+ 4 d \norm{h}_2^2 [\, \tfrac{2}{3}\norm{z}_2^2 + \tfrac{1}{4} \norm{u}_2^2 \,]. \label{eq: descent lemma tech 3}
\end{align}
For the second integral in \eqref{eq: descent lemma tech 1}, note that by inequalities
\[
|\alpha - \beta|^2 \le \alpha^2 + \beta^2 \text{ for } \alpha, \beta \ge 0,
\ \text{ and } \ 
|\sqrt{\alpha^2 + \gamma^2} - \sqrt{\beta^2 + \gamma^2}| \le |\alpha - \beta|, \text{ for } \alpha,\beta,\gamma \in \bb R,
\]
the values of $\q L_\varepsilon(z, v)$ are bounded by 
\begin{align}
\q L_\varepsilon(z, v) 
& = \sum_{j=1}^m \left| \sqrt{|z^T Q_j v|^2 + \varepsilon} - \sqrt{y_j + \varepsilon}\right|^2 
\le \sum_{j=1}^m \left| |z^T Q_j v| - \sqrt{y_j} \right|^2 \nonumber \\
& \le \sum_{j=1}^m |z^T Q_j v|^2 + \norm{y}_1
\le d \norm{z}_2^2 \norm{v}_2^2 + \norm{y}_1 \label{eq: L bound}
\end{align}
Hence, the triangle inequality with \eqref{eq: descent lemma tech 2} gives
\begin{align*}
& 4 \int_0^1 (1 - t)\left[ \q L_\varepsilon(z+tu,v+th) \right]^{1/2} \cdot \left[ \sum_{j=1}^m |u^T Q_j h|^2 \right]^{1/2} dt \\
& \quad \quad \le 4 \sqrt{d} \norm{u}_2 \norm{h}_2 \int_0^1 (1 - t) \sqrt{d \norm{z +tu}_2^2 \norm{v + th}_2^2 + \norm{y}_1} dt \\
& \quad \quad \le 4 \sqrt{d} \norm{u}_2 \norm{h}_2 \int_0^1 (1 - t) \sqrt{d} \norm{z +tu}_2 \norm{v + th}_2  +  \norm{y}_1^{1/2} (1 - t) dt \\
& \quad \quad \le  [ \norm{u}_2^2 + \norm{h}_2^2] \int_0^1 (1 - t) d [ \norm{z +tu}_2^2 + \norm{v + th}_2^2] + 2\sqrt{d} \norm{y}_1^{1/2} (1 - t) dt \\
& \quad \quad \le [ \norm{u}_2^2 + \norm{h}_2^2] \left[ d \, [ \tfrac{2}{3}\norm{z}_2^2 + \tfrac{1}{4} \norm{u}_2^2 + \tfrac{2}{3}\norm{v}_2^2 + \tfrac{1}{4} \norm{h}_2^2] + \sqrt{d} \norm{y}_1^{1/2} \right].
\end{align*}
Returning to \eqref{eq: descent lemma tech 1}, the above inequality and \eqref{eq: descent lemma tech 3} combined yield
\begin{align*}
& \begin{bmatrix} s \\ \conj s \end{bmatrix}^* \int_0^1 (1 - t) \nabla^2 \q L_\varepsilon(z + t u, v +th) d t \begin{bmatrix} s \\ \conj s \end{bmatrix} \\
& \quad \quad \quad \quad \le  d \norm{u}_2^2 \left[ \tfrac{10}{3} \norm{v}_2^2 + \tfrac{5}{4} \norm{h}_2^2 + \tfrac{2}{3} \norm{z}_2^2 + \tfrac{1}{4} \norm{u}_2^2 + \norm{y / d}_1^{1/2} \right] \\
& \quad \quad \quad \quad + d \norm{h}_2^2 \left[ \tfrac{10}{3} \norm{z}_2^2 + \tfrac{5}{4} \norm{u}_2^2 + \tfrac{2}{3} \norm{v}_2^2 + \tfrac{1}{4} \norm{h}_2^2 + \norm{y / d}_1^{1/2} \right].
\end{align*}
Substituting this into the Taylor approximation concludes the proof for $\varepsilon>0$,
\begin{align*}
\q L_\varepsilon(z+u, v + h) 
& \le \q L_\varepsilon(z, v) 
+ 2 \RE( u^* \nabla_z \q L_\varepsilon(z,v) ) 
+ 2 \RE(h^*\nabla_v \q L_\varepsilon(z,v)) \\
& \quad + d \norm{u}_2^2  \left[ \tfrac{10}{3} \norm{v}_2^2 + \tfrac{5}{4} \norm{h}_2^2 + \tfrac{2}{3} \norm{z}_2^2  +  \tfrac{1}{4}\norm{u}_2^2 + \norm{y/d}_1^{1/2} \right] \\
& \quad + d \norm{h}_2^2  \left[ \tfrac{10}{3} \norm{z}_2^2 + \tfrac{5}{4} \norm{u}_2^2 + \tfrac{2}{3} \norm{v}_2^2 +  \tfrac{1}{4}\norm{h}_2^2 + \norm{y/d}_1^{1/2} \right]. 
%& \quad + d ( \norm{u}_2^2 + \norm{h}_2^2 ) \left[ \tfrac{10}{3} (\norm{z}_2^2 + \norm{v}_2^2) + \tfrac{5}{4} (\norm{u}_2^2 + \norm{h}_2^2) + \norm{y + \varepsilon}_1^{1/2} \right]. 
\end{align*}
The case $\varepsilon =0$ follows by taking the limit $\varepsilon \to 0+$.

Turning to $\q J$, we extend the above bound using Lemma 2.5 of \cite{Filbir.2022}. Let $\q T(z,v) := \alpha_T \norm{z}_2^2 + \beta_T \norm{v}_2^2$.
By Lemma 2.5  of \cite{Filbir.2022}, for $s = (u,h)$, we have 
\[
\begin{bmatrix} s \\ \conj s \end{bmatrix}^*
\nabla^2 \q T
\begin{bmatrix} s \\ \conj s \end{bmatrix}
\le 2 \alpha_T \norm{u}_2^2 + 2 \beta_T \norm{h}_2^2.
\]
Applying this error bound to the second-order Taylor expansion \eqref{eq: Taylor} of $\q T$ gives
\begin{align*}
&\q T(z+u, v+h) 
= \q T(z, v) 
+ 2 \RE( u^* \nabla_z \q T(z,v) ) 
+ 2 \RE(h^*\nabla_v \q T(z,v)) \\
& \quad \quad \quad \quad \quad \quad \quad + \begin{bmatrix} s \\ \conj s \end{bmatrix}^*
\int_0^1 (1 - t) \nabla^2 \q T(z + t u, v +th) d t \begin{bmatrix} s \\ \conj s
\end{bmatrix} \\
& \quad \quad \quad \quad \le \q T(z, v) 
+ 2 \RE( u^* \nabla_z \q T(z,v) ) 
+ 2 \RE(h^*\nabla_v \q T(z,v))
+ \alpha_T \norm{u}_2^2 +  \beta_T \norm{h}_2^2.
\end{align*}
Since $\q J = \q L_\varepsilon + \q T$, combining two inequalities together yields the desired result. 
\end{proof}

Note that for our specific $Q_j$ the inequality \eqref{eq: bilinear norm bound} holds and, in general for arbitrary $Q_j$, the sum $\sum_{j=1}^m |z^T Q_j v|^2$ can always be bounded as
\[
\sum_{j=1}^m |z^T Q_j v|^2 
\le \sum_{j=1}^m \norm{z}_2^2 \norm{Q_j}^2 \norm{v}_2^2
\le \norm{z}_2^2 \norm{v}_2^2 \sum_{j=1}^m \norm{Q_j}^2,
\]
where $\norm{\cdot}$ denotes the spectral norm of a matrix.

\subsection{Convergence of gradient descent}

In this section, we derive convergence result for the gradient descent algorithm \eqref{eq: gradient descent}.

\begin{proof}[Proof of \Cref{thm: convergence gradient}.]
Using \Cref{col: descent lemma J} with $z = z^t$, $v = v^t$, $u = \mu_t \nabla_z \q J(z^t,v^t)$, $h =\nu_t \nabla_v \q J(z^t,v^t)$ and the definition \eqref{eq: B} of $B$. We get
\begin{align*}
& \q J(z^{t+1}, v^{t+1}) 
 \le \q J(z^t, v^t) 
- 2 \mu_t \norm{ \nabla_z \q J(z^t,v^t) }_2^2 
- 2 \nu_t \norm{ \nabla_v \q J(z^t,v^t)}_2^2 \\
& \quad \quad \quad \quad + \mu_t \norm{ \nabla_z \q J(z^t,v^t) }_2^2  \left[ \tfrac{1}{3}\mu_t B(z^t,v^t)  +   \tfrac{d}{4} \mu_t^3 \norm{ \nabla_z \q J(z^t,v^t)}_2^2 + \tfrac{5d}{4} \mu_t \nu_t^2 \norm{\nabla_v \q J(z^t,v^t)}_2^2 \right] \\
& \quad \quad \quad \quad + \nu_t \norm{ \nabla_v \q J(z^t,v^t)}_2^2  \left[ \tfrac{1}{3} \nu_t B(z^t,v^t) + \tfrac{5 d}{4} \nu_t \mu_t^2 \norm{ \nabla_z \q J(z^t,v^t) }_2^2  + \tfrac{d}{4} \nu_t^3 \norm{ \nabla_v \q J(z^t,v^t)}_2^2  \right].
\end{align*}
Recall that the step sizes $\mu_t$ and $\nu_t$ are as in \eqref{eq: step size gradient}.
The first case of the minimums yields 
\begin{align*}
\mu_t B(z^t,v^t) \le 1, &  \quad
\nu_t  B(z^t,v^t) \le 1.
\end{align*}
The second case leads to 
\begin{align*}
\tfrac{1}{4} d \mu_t^3 \norm{ \nabla_z \q J(z^t,v^t)}_2^2 \le \tfrac{1}{15} \le \tfrac{1}{3}
\quad \text{ and } \quad 
\tfrac{1}{4} d \nu_t^3 \norm{ \nabla_v \q J(z^t,v^t)}_2^2 \le \tfrac{1}{15} \le \tfrac{1}{3}, 
\end{align*}
From the last inequalities and the last case for $\mu_t$ we deduce that
\begin{align*}
\tfrac{5}{4} d \mu_t \nu_t^2 \norm{\nabla_v \q J(z^t,v^t)}_2^2 
& = (\tfrac{5}{4} d)^{1/3} \mu_t \norm{\nabla_v \q J(z^t,v^t)}_2^{2/3} \cdot \left(\tfrac{5}{4} d \nu_t^3 \norm{\nabla_v \q J(z^t,v^t)}_2^{2} \right)^{2/3} 
\le \tfrac{1}{3},
\end{align*}
and, analogously, $\tfrac{5}{4} \mu_t^2 \nu_t \norm{ \nabla_z \q J(z^t,v^t) }_2^2 \le 1/3$. Combining these inequalities gives the standard descent lemma for gradient descent \cite[Lemma 5.7]{Beck.2017},
\begin{equation}\label{eq: objective decreases gradient}
\q J(z^{t+1}, v^{t+1}) 
\le \q J(z^t, v^t) 
- \mu_t \norm{ \nabla_z \q J(z^t,v^t) }_2^2 
- \nu_t \norm{ \nabla_v \q J(z^t,v^t)}_2^2
\le \q J(z^t, v^t).
\end{equation}
To show the convergence of the gradient, we assume that $\mu_t$ and $\nu_t$ are equal to the minimums in \eqref{eq: step size gradient} and sum up the above inequality for $t =0, \ldots, T-1$, 
\[
\sum_{t=0}^{T-1} \mu_t \norm{ \nabla_z \q J(z^t,v^t) }_2^2 + \nu_t \norm{ \nabla_v \q J(z^t,v^t)}_2^2 
\le \q J(z^0, v^0) - \q J(z^T, v^T) 
\le \q J(z^0, v^0) -  \q J^{\inf}.
\]
This implies that the series $\sum_{t=0}^{\infty} \mu_t \norm{ \nabla_z \q J(z^t,v^t) }_2^2 + \nu_t \norm{ \nabla_v \q J(z^t,v^t)}_2^2$ are convergent and its summand 
\begin{equation}\label{eq: tech gradient 1}
\mu_t \norm{ \nabla_z \q J(z^t,v^t) }_2^2 + \nu_t \norm{ \nabla_v \q J(z^t,v^t)}_2^2 \to 0 \text{ as } t \to \infty.
\end{equation}
Let us show that the gradient 
\begin{equation}\label{eq: def gradient J}
\nabla \q J(z^t,v^t) 
= \begin{bmatrix}
\nabla_z \q J(z^t,v^t) \\
\nabla_v \q J(z^t,v^t) 
\end{bmatrix},
\quad 
\norm{\nabla \q J(z^t,v^t)}_2^2
= \norm{ \nabla_z \q J(z^t,v^t) }_2^2  + \norm{ \nabla_v \q J(z^t,v^t)}_2^2
\end{equation}
vanishes. First, for $\alpha_T,\beta_T >0$ we observe that   
\[
\alpha_T \norm{z^t}_2^2 \le \q L_\varepsilon(z^t,v^t) + \alpha_T \norm{z^t}_2^2 + \beta_T \norm{v^t}_2^2 
= \q J(z^{t}, v^{t}) 
\le \q J(z^0, v^0),
\]  
so that $\norm{z^t}_2^2 \le \q J(z^0, v^0)/\alpha_T$ and, analogously, $\norm{v^t}_2^2 \le \q J(z^0, v^0)/\beta_T$.

Now, we are ready to prove the convergence of the gradient to zero. Let us first consider the case $\norm{ \nabla_z \q J(z^t,v^t) }_2 \ge \norm{ \nabla_v \q J(z^t,v^t)}_2$, which gives
\[
\norm{ \nabla_z \q J(z^t,v^t) }_2^2 
\ge \tfrac{1}{2} \left( \norm{ \nabla_z \q J(z^t,v^t) }_2^2  + \norm{ \nabla_v \q J(z^t,v^t)}_2^2 \right)
= \tfrac{1}{2} \norm{\nabla \q J(z^t,v^t)}_2^2,
\]
\[
\norm{ \nabla_z \q J(z^t,v^t) }_2^{4/3}
= \left( \norm{ \nabla_z \q J(z^t,v^t) }_2^2 \right)^{2/3} 
\ge \tfrac{1}{2^{2/3}} \norm{\nabla \q J(z^t,v^t)}_2^{4/3}.
\]
and 
\[
\frac{\norm{ \nabla_z \q J(z^t,v^t) }_2^{2}}{\norm{ \nabla_v \q J(z^t,v^t) }_2^{2/3}}
\ge \norm{ \nabla_z \q J(z^t,v^t) }_2^{4/3}
\ge \tfrac{1}{2^{2/3}} \norm{\nabla \q J(z^t,v^t)}_2^{4/3}.
\]
Then, we have 
\begin{align*}
\mu_t \norm{ \nabla_z \q J(z^t,v^t) }_2^2 
& \ge \min \left\{ \frac{\norm{ \nabla_z \q J(z^t,v^t) }_2^2}{3 d \left[ (\tfrac{10}{3} \alpha_T^{-1} + \tfrac{10}{3} \beta_T^{-1}) \q J(z^0, v^0) + \norm{y/d}_1^{1/2}\right]  + \max\{ \alpha_T, \beta_T \}  }, \right. \\
& \quad \quad \quad \quad  \left. \frac{4^{\frac{1}{3}} \norm{ \nabla_z \q J(z^t,v^t)}_2^{\frac{4}{3}} }{(15d)^{\frac{1}{3}}}, 
\frac{4^{\frac{1}{3}} \norm{ \nabla_z \q J(z^t,v^t) }_2^2 }{(15d)^{\frac{1}{3}} \norm{\nabla_v \q J(z^t,v^t)}_2^{2/3}} \right\} \\ 
& \ge C^{-1} \min \left\{ \norm{\nabla \q J(z^t,v^t)}_2^2, \norm{\nabla \q J(z^t,v^t)}_2^{4/3} \right\},
\end{align*} 
where $C$ is defined as in \eqref{eq: c gradient}. In the opposite case, $\norm{ \nabla_z \q J(z^t,v^t) }_2 < \norm{ \nabla_v \q J(z^t,v^t)}_2$, exactly the same inequality is satisfied for $\nu_t \norm{ \nabla_v \q J(z^t,v^t)}_2^2$.
Thus, in any case, we get
\begin{align*}
\mu_t \norm{ \nabla_z \q J(z^t,v^t) }_2^2 + \nu_t \norm{ \nabla_v \q J(z^t,v^t)}_2^2
\ge 
C^{-1} \min \left\{ \norm{\nabla \q J(z^t,v^t)}_2^2, \norm{\nabla \q J(z^t,v^t)}_2^{4/3} \right\}.
\end{align*}
Combining it with \eqref{eq: tech gradient 1} gives,
\[
C^{-1} \min \left\{ \norm{\nabla \q J(z^t,v^t)}_2^2, \norm{\nabla \q J(z^t,v^t)}_2^{4/3} \right\} \to 0 \text{ as } t \to \infty,
\]
which is only possible if $\norm{\nabla \q J(z^t,v^t)}_2 \to 0$ as $t \to \infty$. Finally, to quantify the convergence rate we note that
\begin{align*}
& \min \left\{ \min_{t = 0, \ldots, T-1} \norm{\nabla \q J(z^t,v^t)}_2^2, \left[ \min_{t = 0, \ldots, T-1} \norm{\nabla \q J(z^t,v^t)}_2^2 \right]^{2/3} \right\} \\
& \quad = \min_{t = 0, \ldots, T-1} \min \left\{ \norm{\nabla \q J(z^t,v^t)}_2^2, \norm{\nabla \q J(z^t,v^t)}_2^{4/3} \right\} \\
& \quad \le C \min_{t = 0, \ldots, T-1} \left[\mu_t \norm{ \nabla_z \q J(z^t,v^t) }_2^2 + \nu_t \norm{ \nabla_v \q J(z^t,v^t)}_2^2 \right] \\
& \quad \le \frac{C}{T} \sum_{t=0}^{T-1} \left[\mu_t \norm{ \nabla_z \q J(z^t,v^t) }_2^2 + \nu_t \norm{ \nabla_v \q J(z^t,v^t)}_2^2 \right]
\le \frac{C}{T} [\q J(z^0, v^0) -  \q J^{\inf}].
\end{align*}
Thus, using that for all $a,b \ge 0$ the inequality
\[
\min\{a, a^{2/3}\} \le b 
\ \Leftrightarrow \ 
a \le b \text{ or } a \le b^{3/2} 
\ \Leftrightarrow \ 
a \le \max\{ b, b^{3/2}\}
\]
holds, we obtain
\[
\min_{t = 0, \ldots, T-1} \norm{\nabla \q J(z^t,v^t)}_2^2 \le \max \left\{ C T^{-1} [\q J(z^0, v^0) -  \q J^{\inf}], C^{3/2} T^{-3/2} [\q J(z^0, v^0) -  \q J^{\inf}]^{3/2} \right\}.
\]
\end{proof}
\subsection{Convergence of stochastic gradient descent}

In this section, we derive \Cref{thm: convergence stochastic gradient}. The proof follows the steps of Theorem 3.7 of \cite{Melnyk.2022b}. However, the difference between the standard descent lemma and \Cref{col: descent lemma J} affects the proof, which has to be carefully adjusted. 
%As the stochastic gradients are random, the proof involves some notions from probability theory and theory of stochastic process. 

Let us start with a brief overview of the proof. The idea is similar to \Cref{thm: convergence gradient}. That is, we would like to show that the sequence $\{ \q J(z^t, v^t) \}_{t \ge 0}$ is convergent and the sequence $\{ \norm{\nabla \q J(z^t, v^t)}_2 \}_{t \ge 0}$ vanishes. Yet, as these sequences are determined by indices $\{r^{t,k} \}_{t = 0, k= 1}^{\infty, K}$, they are random. 

Our goal is to apply the next result by Robbins and Siegmund \cite{Robbins.1971} with random sequences 
\[
Y_t = \q J(z^t,v^t) - \q J^{\inf}, 
\quad \text{and} \quad
X_t = \mu_t \norm{\nabla_z \q J(z^t,v^t)}_2^2 + \nu_t \norm{\nabla_z \q J(z^t,v^t)}_2^2.
\]
\begin{theorem}[{\cite[Theorem 1]{Robbins.1971}}]\label{thm: near-supermartingale convergence}
Let $\{X_t\}_{t \ge 0}$, $\{Y_t\}_{t \ge 0}$ and $\{Z_t\}_{t \ge 0}$ be three sequences of nonnegative random variables that are adapted to a filtration $\{ \q F_t \}$. Let $\eta_t$ be a sequence of nonnegative real numbers such that $\sum_{t=0}^\infty \eta_t < \infty$. Suppose that 
\begin{equation}\label{eq: near-supermartingale condition}
\bb E[Y_{t+1}\,|\, \q F_t \,] \le (1 + \eta_t) Y_t - X_t + Z_t, \quad \text{for all } t \ge 0,
\end{equation}
and $\sum_{t = 0}^\infty Z_t < \infty$ almost surely. Then, $\sum_{t = 0}^\infty X_t < \infty$ almost surely and $Y_t$ converges almost surely. 
\end{theorem}
Thus, the proof is split into three parts. In the first part, we derive \eqref{eq: near-supermartingale condition} based on \Cref{col: descent lemma J}. In the second part, we apply \Cref{thm: near-supermartingale convergence} to obtain convergence of $\{ \q J(z^t, v^t) \}_{t \ge 0}$. Then, we deduce that with a proper choice of step sizes, there exists a subsequence $\{t_k\}_{k \ge 0}$ such that $\norm{\nabla \q J(z^{t_k},v^{t_k})}_2 \to 0$ as $k \to \infty$. The third part follows ideas of \cite{Orabona.2020} to deduce $\norm{\nabla \q J(z^{t},v^{t})}_2 \to 0$ using Lipschitz continuity of the gradient. 

Throughout this section, steps of the proof are separated into lemmas. For each, the requirements on the step sizes are stated separately to highlight how they change as the proof progresses. 

\subsubsection{First half: preparations for \Cref{thm: near-supermartingale convergence}}
Let us start by establishing the inequality \eqref{eq: near-supermartingale condition}. The first step is to combine \Cref{col: descent lemma J} with updates \eqref{eq: stochastic gradient descent}. 
\begin{lemma}\label{l: descent lemma stochastic first}
Let $\varepsilon, \alpha_T, \beta_T \ge 0$ and $0 \le \theta < 1$. If the step sizes satisfy
\begin{align} 
\mu_t, \nu_t & \le \min \left\{ B^{-\frac{1}{1 - \theta}}(z^t,v^t), \left( \tfrac{15d}{4} \right)^{-\frac{1}{3-\theta}} \norm{ g_z (z^t,v^t)}_2^{-\frac{2}{3 - \theta}},  
\left( \tfrac{15 d}{4} \right)^{-\frac{1}{3-\theta}} \norm{g_v (z^t,v^t)}_2^{-\frac{2}{3 - \theta}} \right\},  \label{eq: step sizes stochastic descent} %\\
%\nu_t & \le \min \left\{ \tfrac{1}{3} d^{-1}  \left[ \tfrac{10}{3} \norm{z^t}_2^2  + \tfrac{2}{3} \norm{v^t}_2^2 + \norm{y/d}_1^{1/2}  + \beta_T \right]^{-1}, \right. \nonumber \\
%& \quad \quad \quad  \left. \left( \tfrac{4}{3} \right)^{1/3} d^{-1/3} \norm{ \nabla_v \q J(z^t,v^t)}_2^{-2/3},  \tfrac{4^{1/3}}{5 \cdot 3^{1/3}} d^{-1/3} \norm{\nabla_z \q J(z^t,v^t)}_2^{-2/3}  \right\}. \nonumber
\end{align}
then we have
\begin{align*}
\q J(z^{t+1},v^{t+1}) 
& \le \q J(z^t, v^t) 
- 2 \mu_t \RE( g_z^*(z^t,v^t) \nabla_z \q J(z^t,v^t) ) 
- 2 \nu_t \RE( g_v^*(z^t,v^t) \nabla_v \q J(z^t,v^t) ) \\
& \quad + \mu_t^{1 + \theta} \norm{g_z (z^t,v^t)}_2^2 
+ \nu_t^{1 + \theta} \norm{g_v (z^t,v^t)}_2^2.
\end{align*}
\end{lemma}
\begin{proof}
The proof is analogous to the proof of \Cref{thm: convergence gradient} from the beginning to the equation \eqref{eq: objective decreases gradient}.
\end{proof}

Comparing the result of \Cref{l: descent lemma stochastic first} with \eqref{eq: near-supermartingale condition}, the next step is to compute the conditional expectation, which would require the following properties of the stochastic gradient.

\begin{lemma}\label{l: stochastic gradient properties 2}
Consider the sequences $\{ z^t \}_{t \ge 0}$ and $\{ v^t \}_{t \ge 0}$ defined by \eqref{eq: stochastic gradient descent} with stochastic gradients $g(z^t, v^t)$ given by \eqref{eq: stochastic gradient}. Then, we have
$ \bb E[ \, g(z^t,v^t) \, | \, \q F_t \, ] = \nabla \q J(z^t, v^t)$, 
and
\begin{align*}
\bb E[ \, \norm{g_z(z^t,v^t)}_2^2 \, | \, \q F_t \, ] 
& \le 
\gamma_t^z [\q J(z^t, v^t) - \q J^{\inf}] + \rho \norm{\nabla_z \q J(z^t, v^t)}_2^2 + \delta_t^z, \\
\bb E[ \, \norm{g_v(z^t,v^t)}_2^2 \, | \, \q F_t \, ] 
& \le 
\gamma_t^v [\q J(z^t, v^t) - \q J^{\inf}] + \rho \norm{\nabla_v \q J(z^t, v^t)}_2^2 + \delta_t^v,
\end{align*}
where $\q J^{\inf}$ is defined in \eqref{eq: J inf} and 
\begin{align*}
& \gamma_t^z := \frac{d \norm{v^t}_2^2}{K \min_{r \in \q R} p_r} + \frac{\alpha_T}{K},
\quad
\rho := 1 - \frac{1}{K},
\quad
\delta_t^z := \gamma_t^z [ \q J^{\inf} - \sum_{r \in \q R} \inf_{z,v \in \bb C^d} \q J_{r}(z,v)],
\\
& \gamma_t^v := \frac{d \norm{z^t}_2^2}{K \min_{r \in \q R} p_r} + \frac{\beta_T}{K},
\quad
\delta_t^v:= \gamma_t^v [ \q J^{\inf} - \sum_{r \in \q R} \inf_{z,v \in \bb C^d} \q J_{r}(z,v)],
\end{align*}
\end{lemma}
\begin{proof}
The first inequality is shown by direct computation,
\begin{align*}
\bb E[ \, g(z^t,v^t) \, | \, \q F_t \, ] 
= \frac{1}{K} \sum_{k=1}^K \bb E \left[ \, \frac{\nabla \q J_{ r^{t,k} }(z^t, v^t)}{ p_{r^{t,k}} } \, \bigg| \, \q F_t \, \right]
= \frac{1}{K} \sum_{k=1}^K \sum_{r \in \q R} \nabla \q J_{r}(z^t, v^t) 
= \q J(z^t,v^t).
\end{align*}
For the last two inequalities, we adjust the proof of Proposition 3 in \cite{Khaled.2023}. That is, the equation (41) in \cite{Khaled.2023} gives
\[
\bb E[ \, \norm{g_z(z^t,v^t)}_2^2 \, | \, \q F_t \, ] 
\le 
\rho \norm{\nabla_z \q J(z^t, v^t)}_2^2 + \sum_{r \in \q R} (K p_r)^{-1}\norm{\nabla_z \q J_r (z^t, v^t)}_2^2.
\]  
Note that $\q J_r$ can be seen as a copy of $\q J$ with $\q R^\prime = \{r\}$ and $\alpha_T^\prime = \alpha_T p_r$. Thus, by \Cref{thm: previous results}, we have
\begin{align*}
\norm{\nabla_z \q J_r (z^t, v^t)}_2^2 
& \le ( d \norm{v^t}_2^2 + \alpha_T p_r) [\q J_r(z^t, v^t) - \q J_r(z^{t+1}, v^{t})] \\
& \le ( d \norm{v^t}_2^2 + \alpha_T p_r) [\q J_r(z^t, v^t) - \inf_{z,v \in \bb C^d} \q J_r(z,v)].
\end{align*}
This leads to
\begin{align*}
&\bb E[ \, \norm{g_z(z^t,v^t)}_2^2 \, | \, \q F_t \, ] 
 \le  \rho \norm{\nabla_z \q J(z^t, v^t)}_2^2 
+ \sum_{r \in \q R} \frac{d \norm{v^t}_2^2 + \alpha_T p_r}{K p_r} [\q J_r(z^t, v^t) - \inf_{z,v \in \bb C^d} \q J_r(z,v)] \\
& \quad \quad \quad \quad \le \rho \norm{\nabla_z \q J(z^t, v^t)}_2^2  + \left[ \frac{d \norm{v^t}_2^2}{K \min_{r \in \q R} p_r} + \frac{\alpha_T}{K} \right] [\q J(z^t,v^t) - \sum_{r \in \q R} \inf_{z,v \in \bb C^d} \q J_r(z,v) ] \\
& \quad \quad \quad \quad = \rho \norm{\nabla_z \q J(z^t, v^t)}_2^2  + \gamma_t^z [\q J(z^t,v^t) - \q J^{\inf}] + \delta_t^z.
\end{align*}
The last inequality is analogous.
\end{proof}
At this point, we also prove \Cref{l: stochastic gradient properties 1} as we would require it for the next step.
\begin{proof}[Proof of \Cref{l: stochastic gradient properties 1}.]
Using the definition of the stochastic gradient \eqref{eq: stochastic gradient}, we observe that
\[
\norm{g_z(z,v)}_2 
= \norm{ \frac{1}{K} \sum_{k=1}^K p_{r^{k}}^{-1} \nabla_z \q J_{ r^{k} }(z, v)  }_2
\le \frac{1}{K} \sum_{k=1}^K p_{r^{k}}^{-1} \norm{\nabla_z \q J_{ r^{k} }(z, v)}_2.
\] 
Let us bound $\norm{\nabla_z \q J_{r}(z, v)}_2$ for each $r \in \q R$ and $z,v \in \bb C^d$. Separating regularizer gives
\[
\norm{\nabla_z \q J_{r}(z, v)}_2 
\le \norm{\nabla_z \q L_{r,\varepsilon}(z, v)}_2 + \alpha_T p_r \norm{z}_2
= \max_{\norm{u}_2 = 1} | u^* \nabla_z \q L_{r,\varepsilon}(z, v) | + \alpha_T p_r \norm{z}_2.
\]
By construction, $\q L_{r,\varepsilon}$ can be seen as a copy of $\q L_{\varepsilon}$ with the set $\q R^\prime = \{r\}$. Thus, it can be written in the form \eqref{eq: loss general} with matrices $Q_{r,k}$, $k=1,\ldots, d$ as in \eqref{eq: matrix Q definition}.
Therefore, the absolute value above can be bounded by repeating the computations in \eqref{eq: gradient bound later} and \eqref{eq: bilinear norm bound}, which yields
\begin{align*}
| u^* \nabla_z \q L_{r,\varepsilon}(z, v) |
& = \left| \sum_{k=1}^d \left[1 - \frac{\sqrt{y_{r,k} + \varepsilon}}{\sqrt{|z^T Q_{r,k} v|^2 + \varepsilon} } \right] z^T Q_{r,k} v \cdot \conj{u^T Q_{r,k} v} \ \right| \\
& \le \left[ \q L_{r,\varepsilon}(z,v) \right]^{1/2} \cdot \left[ \sum_{k=1}^d |u^T Q_{r,k} v|^2 \right]^{1/2}
\le d^{1/2} \left[ \q L_{r,\varepsilon}(z,v) \right]^{1/2} \norm{v}_2 \norm{u}_2.
\end{align*}
Consequently, we get
\[
\norm{\nabla_z \q J_{r}(z, v)}_2 
\le d^{1/2} \left[ \q L_{r,\varepsilon}(z,v) \right]^{1/2} \norm{v}_2 + \alpha_T p_r \norm{z}_2
\] 
and 
\begin{align*}
\norm{g_z(z,v)}_2 
& \le \frac{1}{K} \sum_{k=1}^K p_{r^{k}}^{-1}  d^{1/2}  \norm{v}_2 \left[ \q L_{r^{k},\varepsilon}(z,v) \right]^{1/2} + \frac{1}{K} \sum_{k=1}^K \alpha_T \norm{z}_2 \\
& \le \frac{d^{1/2}  \norm{v}_2 }{K} \left[ \sum_{k=1}^K \q L_{r^{k},\varepsilon}(z,v) \right]^{1/2} \left[ \sum_{k=1}^K p_{r^{k}}^{-2} \right]^{1/2} + \alpha_T \norm{z}_2 \\
& \le \frac{d^{1/2} \norm{v}_2 }{\sqrt K \min_{r \in \q R} p_r} [\q L_{\varepsilon}(z,v) ]^{1/2} + \alpha_T \norm{z}_2.
\end{align*}
Then, using \eqref{eq: L bound}, we conclude that
\begin{align*}
\norm{g_z(z,v)}_2 
& \le \frac{d^{1/2} \norm{v}_2 }{K^{1/2} \min_{r \in \q R} p_r} [d \norm{z}_2^2  \norm{v}_2^2 + \norm{y}_1 ]^{1/2} + \alpha_T \norm{z}_2 \\
& \le \frac{d \norm{v}_2}{K^{1/2} \min_{r \in \q R} p_r} [\norm{z}_2  \norm{v}_2 +  \norm{y/d}_1^{1/2} ] + \alpha_T \norm{z}_2.
\end{align*}
Multiplying two sides by $(\tfrac{15}{4}d)^{1/2}$ concludes the proof for the first bound. The second bound is analogous.
\end{proof}

Now, we combine the results of the previous two lemmas to get the bound of the form \eqref{eq: near-supermartingale condition}.

\begin{lemma}\label{l: descent lemma stochastic second}
Let $\varepsilon, \alpha_T, \beta_T \ge 0$ and $0 \le \theta < 1$. Assume that the step sizes are adapted to the filtration $\{ \q F_t \}_{t \ge 0}$ and satisfy 
\begin{align} 
\mu_t, \nu_t & \le \min \left\{ B^{-\frac{1}{1 - \theta}}(z^t,v^t),  
B_z^{-\frac{2}{3 - \theta}} (z^t,v^t), B_v^{-\frac{2}{3 -\theta}} (z^t,v^t),  (1 - \tfrac{1}{K})^{-1/\theta} \right\}. \label{eq: step sizes stochastic descent 2}
\end{align}
Then, we have
\begin{align*}
\bb E[ \, \q J(z^{t+1},v^{t+1}) - \q J^{\inf} \, | \, \q F_t \, ]
& \le (1 + \gamma_t^z \mu_t^{1 + \theta} + \gamma_t^v \nu_t^{1 + \theta}) [\q J(z^t, v^t) - \q J^{\inf}] \\
& \quad - \mu_t \norm{ \nabla_z \q J(z^t,v^t) }_2^2 - \nu_t \norm{ \nabla_v \q J(z^t,v^t) }_2^2 + \mu_t^{1 + \theta}\delta_t^z + \nu_t^{1 + \theta} \delta_t^v, 
\end{align*}
with $\gamma_t^z,\gamma_t^v,\delta_t^z, \delta_t^v$ as in \Cref{l: stochastic gradient properties 2}.
\end{lemma}
\begin{proof}
Note that the inequality \eqref{eq: step sizes stochastic descent 2} combined with \Cref{l: stochastic gradient properties 1} implies that the step sizes satisfy \eqref{eq: step sizes stochastic descent} and, therefore, \Cref{l: descent lemma stochastic first} can be used. Taking the conditional expectation with respect to $\q F_t$ and using the fact that step sizes are adapted gives
\begin{align*} 
\bb E[ \, \q J(z^{t+1},v^{t+1})  \, | \, \q F_t \, ]
& \le \q J(z^t, v^t) 
- 2 \mu_t \RE( \bb E[ \, g_z^*(z^t,v^t) \, | \, \q F_t \, ] \nabla_z \q J(z^t,v^t) ) 
 \\
& \quad - 2 \nu_t \RE( \bb E[ \, g_v^*(z^t,v^t)  \, | \, \q F_t \, ]  \nabla_v \q J(z^t,v^t) )  \\
& \quad + \mu_t^{1 + \theta} \bb E[ \, \norm{g_z (z^t,v^t)}_2^2 \, | \, \q F_t \, ]
+ \nu_t^{1 + \theta} \bb E[ \, \norm{g_v (z^t,v^t)}_2^2 \, | \, \q F_t \, ].
\end{align*}
Next, \Cref{l: stochastic gradient properties 2} is applied,
\begin{align*} 
\bb E[ \, \q J(z^{t+1},v^{t+1})  \, | \, \q F_t \, ]
& \le \q J(z^t, v^t) 
- 2 \mu_t \norm{ \nabla_z \q J(z^t,v^t) }_2^2  - 2 \nu_t \norm{ \nabla_v \q J(z^t,v^t) }_2^2 \\
& \quad + \mu_t^{1 + \theta} [\gamma_t^z [\q J(z^t, v^t) - \q J^{\inf}] + \rho \norm{\nabla_z \q J(z^t, v^t)}_2^2 + \delta_t^z] \\
& \quad + \nu_t^{1 + \theta} [\gamma_t^v [\q J(z^t, v^t) - \q J^{\inf}] + \rho \norm{\nabla_z \q J(z^t, v^t)}_2^2 + \delta_t^v].
\end{align*}
If $\rho = 1-K^{-1} =0$, the corresponding terms can be discarded and a trivial bound gives $-2\mu_t \le -\mu_t$ and $-2\nu_t \le -\nu_t$. Otherwise, by \eqref{eq: step sizes stochastic descent 2}, $\mu_t^\theta \rho \le 1$, $\nu_t^\theta \rho \le 1$, which gives
\[
- 2 \mu_t + \mu_t^{1 + \theta} \rho \le - \mu_t
\quad \text{and} \quad
- 2 \nu_t + \nu_t^{1 + \theta} \rho \le - \nu_t.
\] 
Subtracting $\q J^{\inf}$ on both sides and grouping the terms concludes the proof. 
\end{proof}
Note that if the step sizes were not adapted, it would be not possible to separate them from the stochastic gradients when computing the conditional expectation. Furthermore, to ensure that $\mu_t$ and $\nu_t$ do not depend on $g(z^t,v^t)$ explicitly as in \eqref{eq: step sizes stochastic descent} we used the bounds $B_z$ and $B_v$ introduced in \Cref{l: stochastic gradient properties 1}. In this way, the right-hand side of \eqref{eq: step sizes stochastic descent 2} is also adapted with $\{\q F_t\}_{t \ge 0}$. 

\subsubsection{Applying \Cref{thm: near-supermartingale convergence} and establishing convergence}

Finally, \Cref{thm: near-supermartingale convergence} can be applied, which gives us the first claim of \Cref{thm: convergence stochastic gradient}.

\begin{lemma}\label{l: convergence stochastic first}
Let $\varepsilon, \alpha_T, \beta_T \ge 0$, $0 \le \theta < 1$ and $\kappa < \theta/(1 + \theta)$. Assume that the step sizes are adapted to the filtration $\{ \q F_t \}_{t \ge 0}$ and satisfy $\mu_t, \nu_t \le \mu_t^{\max}$ with $\mu_t^{\max}$ as in \eqref{eq: mu max}. Then, the sequence $\q J(z^{t+1},v^{t+1})$ converges a.s. Furthermore, 
\[
\sum_{t \ge 0} \mu_t \norm{ \nabla_z \q J(z^t,v^t) }_2^2 + \nu_t \norm{ \nabla_v \q J(z^t,v^t) }_2^2 < \infty \quad \text{a.s.}
\]
\end{lemma}
\begin{proof}
Since   $\mu_t, \nu_t \le \mu_t^{\max}$ and $\kappa <\theta/(1+\theta) \le 1$, it holds that $(1+t)^{-1+\kappa} \le 1$ for $t \ge 0$ and the condition \eqref{eq: step sizes stochastic descent 2} is satisfied. Hence, \Cref{l: descent lemma stochastic second} can be used. As it was mentioned before, our goal is to apply \Cref{thm: near-supermartingale convergence} for sequences $Y_t = \q J(z^t, v^t) - \q J^{\inf}$ and $X_t = \mu_t \norm{ \nabla_z \q J(z^t,v^t) }_2^2 + \nu_t \norm{ \nabla_v \q J(z^t,v^t) }_2^2$. By looking at \Cref{l: descent lemma stochastic second}, it is only natural to pick $\eta_t$ as $\gamma_t^z \mu_t^{1 + \theta} + \gamma_t^v \nu_t^{1+\theta}$ and $Z_t$ as $\delta_t^z \mu_t^{1 + \theta} + \delta_t^v \nu_t^{1+\theta}$. However, the sequence $\eta_t$ has to be deterministic, which is not true for $\gamma_t^z \mu_t^{1 + \theta} + \gamma_t^v \nu_t^{1+\theta}$ as they depend on $z^t$ and $v^t$. Thus, we first bound $\gamma_t^z \mu_t^{1 + \theta} + \gamma_t^v \nu_t^{1+\theta}$ by a deterministic term. That is, by \eqref{eq: mu max} we have
\begin{align*}
&\gamma_t^z \mu_t^{1 + \theta} + \gamma_t^v \nu_t^{1+\theta}
\le (\gamma_t^z  + \gamma_t^v) (\mu_t^{\max})^{1 + \theta}
\le \frac{\gamma_t^z + \gamma_t^v}{(1+t)^{(1-\kappa)(1+\theta)} B^{\frac{1 + \theta}{1-\theta}}(z^t,v^t)}\\
& \quad \le \frac{d (\norm{z^t}_2^2 + \norm{v^t}_2^2) (\min_{r \in \q R} p_r)^{-1} + \alpha_T + \beta_T }{(1+t)^{(1-\kappa)(1+\theta)}  K \left[  3 \max\{\alpha_T,\beta_T\} + 3 d \left( \tfrac{10}{3} \norm{z^t}_2^2  + \tfrac{10}{3} \norm{v^t}_2^2 + \norm{y/d}_1^{1/2} \right) \right]^{\frac{1 + \theta}{1 - \theta}}}.
\end{align*}
We can split the power $(1 +\theta)/(1 - \theta)$ in the denominator as $1$ and $2\theta/(1- \theta)$. The first term gives us 
\begin{align*}
& \frac{d (\norm{z^t}_2^2 + \norm{v^t}_2^2) (\min_{r \in \q R} p_r)^{-1} + \alpha_T + \beta_T}{3d \left( \tfrac{10}{3} \norm{z^t}_2^2  + \tfrac{10}{3} \norm{v^t}_2^2 + \norm{y/d}_1^{1/2} \right) + 3\max\{\alpha_T,\beta_T \} } \\
&\quad \le \frac{10 d  (\norm{z^t}_2^2 + \norm{v^t}_2^2) + 3 \max\{ \alpha_T,\beta_T \} \min_{r \in \q R} p_r }{ 1.5 \min_{r \in \q R} p_r  \left[ 10 d \norm{z^t}_2^2  + 10 d \norm{v^t}_2^2 + 3 d \norm{y/d}_1^{1/2}  + 3 d \max\{\alpha_T,\beta_T\} \right]} \\
&\quad  \le \frac{1}{1.5 \min_{r \in \q R} p_r},
\end{align*}
where we used in the last inequality that $p_r \le 1$ for all $r \in \q R$. For the second term, we observe that 
\begin{equation}\label{eq: mu first part bound}
3d\left( \tfrac{10}{3} \norm{z^t}_2^2  + \tfrac{10}{3} \norm{v^t}_2^2 + \norm{y/d}_1^{1/2} \right)  + 3\max\{\alpha_T,\beta_T\}
\ge 3 \sqrt d \norm{y}_1^{1/2}  + 3\max\{\alpha_T,\beta_T\}.
\end{equation}
Combining these two bounds yields
\begin{equation}\label{eq: alpha mu bound}
\gamma_t^z \mu_t^{1 + \theta} + \gamma_t^v \nu_t^{1+\theta} 
\le \frac{(1+t)^{-(1 - \kappa)(1 + \theta)}}{ 1.5 \min_{r \in \q R} p_r [3 \sqrt{d} \norm{y}_1^{1/2}  + 3 \max\{\alpha_T,\beta_T\}]^{\frac{2\theta}{1-\theta}} } =: \eta_t.
\end{equation}
Note that the denominator of $\eta_t$ is a constant. Furthermore, since $\kappa < \theta/(1 + \theta)$, we have
\[
(1 - \kappa)(1 + \theta)
= 1 + \theta - \kappa(1 + \theta) > 1+ \theta - \theta = 1,
\]
and the series $\sum_{t \ge 0} \eta_t$ are convergent. The sequence $Z_t := \delta_t^z \mu_t^{1 + \theta} + \delta_t^v \nu_t^{1+\theta}$ is adapted to $\{\q F_t \}_{t \ge 0}$ and admits
\begin{equation}\label{eq: delta mu bound}
Z_t 
= (\gamma_t^z \mu_t^{1 + \theta} + \gamma_t^v \nu_t^{1+\theta}) [ \q J^{\inf} - \sum_{r \in \q R} \inf_{z,v \in \bb C^d} \q J_{r}(z,v)]
\le \eta_t [ \q J^{\inf} - \sum_{r \in \q R} \inf_{z,v \in \bb C^d} \q J_{r}(z,v)].
\end{equation}
Thus, the condition $\sum_{t \ge 0} Z_t < \infty$ is also satisfied. Also, by \Cref{l: descent lemma stochastic second}, we have
\begin{align*}
\bb E[ \, Y_{t+1} \, | \, \q F_t \, ]
& \le (1 + \gamma_t^z \mu_t^{1 + \theta} + \gamma_t^v \nu_t^{1 + \theta}) Y_t - X_t + Z_t 
\le (1 + \eta_t) Y_t - X_t + Z_t. 
\end{align*}
Hence, by \Cref{thm: near-supermartingale convergence}, $\q J(z^t, v^t) - \q J^{\inf}$ converges a.s. and $\sum_{t \ge 0} \mu_t \norm{ \nabla_z \q J(z^t,v^t) }_2^2 + \nu_t \norm{ \nabla_v \q J(z^t,v^t) }_2^2 < \infty$ a.s. This implies that $\q J(z^t, v^t)$ converges a.s. as well. 
\end{proof}

The next step is to shift our focus on the convergence of the gradient and to prove the second claim of \Cref{thm: convergence stochastic gradient}.

\begin{lemma}\label{l: convergence stochastic second}
Let $\varepsilon \ge 0$, $\alpha_T, \beta_T > 0$, $0 < \theta < 1$ and $0 \le \kappa < \theta/(1 + \theta)$.  Consider the step sizes $\mu_t = \mu  \cdot \mu_t^{\max}$ and $\nu_t = \nu \cdot \mu_t^{\max}$ for some $0 < \mu,\nu \le 1$ and $\mu_t^{\max}$ as in \eqref{eq: mu max}. Then we have
\[
\min_{t=0,\ldots, T -1} \norm{\nabla \q J(z^t,v^t)}_2^2 
\le
\begin{cases}
\frac{\kappa C_2}{\min\{\mu,\nu\}  [(1+T)^\kappa - 1] }, & \kappa > 0, \\
\frac{C_2}{\min\{\mu,\nu\} \ln (1+T) }, & \kappa = 0,
\end{cases}
\quad \text{a.s.,}
\]
and 
\[
\inf_{t \ge 0} \norm{\nabla \q J(z^t,v^t)}_2^2  = 0 \quad \text{a.s.}
\]
\end{lemma}
\begin{proof}
Since $\mu,\nu \le 1$, the step sizes satisfy conditions of \Cref{l: convergence stochastic first}, which gives 
\[
\sum_{t \ge 0} \mu_t \norm{ \nabla_z \q J(z^t,v^t) }_2^2 + \nu_t \norm{ \nabla_v \q J(z^t,v^t) }_2^2 =: C_3 < \infty \quad \text{a.s.}
\]
Thus, we get 
\begin{align*}
\min_{t=0,\ldots, T -1} \norm{\nabla \q J(z^t,v^t)}_2^2
& \le \frac{1}{\min\{\mu,\nu\}} \min_{t=0,\ldots, T -1} \left[ \mu \norm{ \nabla_z \q J(z^t,v^t) }_2^2 + \nu \norm{ \nabla_v \q J(z^t,v^t) }_2^2 \right] \\
& \le \frac{\sum_{t = 0}^{T-1} \mu_t \norm{ \nabla_z \q J(z^t,v^t) }_2^2 + \nu_t \norm{ \nabla_v \q J(z^t,v^t) }_2^2}{\min\{\mu,\nu\} \sum_{t = 0}^{T-1} \mu_t^{\max} } \\
& 
%\le \frac{\sum_{t \ge 0} \mu_t \norm{ \nabla_z \q J(z^t,v^t) }_2^2 + \nu_t \norm{ \nabla_v \q J(z^t,v^t) }_2^2}{\min\{\mu,\nu\} \sum_{t = 0}^{T-1} \mu_t^{\max} } 
= \frac{C_3}{\min\{\mu,\nu\} \sum_{t = 0}^{T-1} \mu_t^{\max} }.
\end{align*}
Let us show that the series $\sum_{t = 0}^{T-1} \mu_t^{\max}$ diverges a.s. By \Cref{l: convergence stochastic first}, we also obtained that $\{ \q J(z^t, v^t) \}_{t \ge 0}$ converges a.s. Hence, it is bounded and let us denote the upper bound by $\q J^{\sup},$ such that $\q J^{\sup} < +\infty$ a.s. Since $\alpha_T > 0$, we get
\begin{equation}\label{eq: norm bounded stochastic}
\alpha_T \norm{z^t}_2^2 
\le \q L_\varepsilon(z^t,v^t) + \alpha_T \norm{z^t}_2^2 + \beta_T \norm{v^t}_2^2 = \q J(z^t, v^t) \le \q J^{\sup},
\end{equation}
so that $\norm{z^t}_2^2 \le \alpha_T^{-1}\q J^{\sup}$ and, analogously $\norm{v^t}_2^2 \le \beta_T^{-1}\q J^{\sup}$. Consequently, we can bound
\begin{align*}
B(z^t,v^t) & = 3d\left( \tfrac{10}{3} \norm{z^t}_2^2  + \tfrac{10}{3} \norm{v^t}_2^2 + \norm{y/d}_1^{1/2} \right)  + 3\max\{\alpha_T,\beta_T\} \\
& \le d \left( 10 (\alpha_T^{-1} + \beta_T^{-1}) \q J^{\sup} + 3 \norm{y/d}_1^{1/2} \right) + 3\max\{\alpha_T,\beta_T\} =: C_{4,1},
\end{align*}
and 
\begin{align*}
& B_z(z^t,v^t)  = (\tfrac{15}{4} d)^{1/2} \left[ \frac{d \norm{v^t}_2}{\sqrt K \min_{r \in \q R} p_r} \left( \norm{z^t}_2  \norm{v^t}_2 +  \norm{y/d}_1^{1/2} \right) + \alpha_T \norm{z^t}_2 \right]  \\
& \quad  \le (\tfrac{15}{4} d)^{1/2} \left[ \frac{d \sqrt{\q J^{\sup}}}{\sqrt{\beta_T K} \min_{r \in \q R} p_r} [\alpha_T^{-1/2} \beta_T^{-1/2} \q J^{\sup}  +  \norm{y/d}_1^{1/2} ] + \alpha_T^{1/2} \sqrt{\q J^{\sup}} \right] =: C_{4,2}, \\
& B_v(z^t,v^t) = (\tfrac{15}{4} d)^{1/2} \left[ \frac{d \norm{z^t}_2}{\sqrt K \min_{r \in \q R} p_r} \left( \norm{z^t}_2  \norm{v^t}_2 +  \norm{y/d}_1^{1/2} \right) + \beta_T \norm{v^t}_2 \right] \\
& \quad  \le (\tfrac{15}{4} d)^{1/2} \left[ \frac{d \sqrt{\q J^{\sup}}}{\sqrt{\alpha_T K} \min_{r \in \q R} p_r} \left( \alpha_T^{-1/2} \beta_T^{-1/2} \q J^{\sup}  +  \norm{y/d}_1^{1/2} \right) + \beta_T^{1/2} \sqrt{\q J^{\sup}} \right]  =: C_{4,3}.
\end{align*}
Hence,
\begin{align}
\mu_t^{\max} 
& \ge \min \left\{ (1+t)^{-1 + \kappa} C_{4,1}^{-\frac{1}{1 - \theta}}, C_{4,2}^{-\frac{2}{3 - \theta}}, C_{4,3}^{-\frac{2}{3 -\theta}}, (1 - \tfrac{1}{K})^{-1/\theta} \right\} \label{eq: mu diverges} \\
& \ge (1+t)^{-1 + \kappa} \min \left\{ C_{4,1}^{-\frac{1}{1 - \theta}}, C_{4,2}^{-\frac{2}{3 - \theta}}, C_{4,3}^{-\frac{2}{3 -\theta}}, (1 - \tfrac{1}{K})^{-1/\theta} \right\} =: (1+t)^{-1 + \kappa} C_{4} > 0 \text{ a.s.}, \nonumber
\end{align}
and
\begin{equation}\label{eq: min bound tech}
\min_{t=0,\ldots, T -1} \norm{\nabla \q J(z^t,v^t)}_2^2
\le \frac{C_3}{C_4 \min\{\mu,\nu\} \sum_{t = 0}^{T-1} (1+t)^{-1 + \kappa} }.
\end{equation}
As $0 \le \kappa$, the inequality $1 - \kappa \le 1$ holds and the series $\sum_{t \ge 0} (1+t)^{-1 + \kappa}$ diverges. This implies that 
\[
\inf_{t \ge 0} \norm{\nabla \q J(z^t,v^t)}_2^2 = 0.
\]
In order to transform the bound in \eqref{eq: min bound tech}  in terms of $T$, we consider cases $\kappa = 0$ and $\kappa > 0$ separately. For the case $\kappa >0$, we have 
\begin{align*}
\sum_{t = 0}^{T-1}  (1+t)^{-1+ \kappa}
& = \sum_{t = 0}^{T-1} \int_t^{t+1} (1+t)^{-1+ \kappa} ds
\ge \sum_{t = 0}^{T-1} \int_{t}^{t+1} (1+s)^{-1+ \kappa} ds \\
& = \int_0^{T} (1+s)^{-1+ \kappa} ds 
= \kappa^{-1} \left[ (1+T)^\kappa - 1 \right].
\end{align*} 
Therefore, we have
\[
\min_{t = 0, \ldots, T-1} \norm{\nabla \q J(z^t,v^t)}_2^2 
\le \frac{\kappa C_3}{C_4 \min\{\mu,\nu\}  [(1+T)^\kappa - 1] }.
\]
If $\kappa = 0$, then
\[
\sum_{t = 0}^{T-1}  (1+t)^{-1}
\ge \int_0^{T} (1+s)^{-1} ds 
= \ln(1 + T)
\]
and
\[
\min_{t = 0, \ldots, T-1} \norm{\nabla \q J(z^t, v^t)}_2^2 
\le \frac{C_3}{C_4 \min\{\mu,\nu\} \ln(1+T)}.
\]
Setting $C_2  = C_3 / C_4$, which is finite a.s., concludes the proof.
\end{proof}

Note that while $\theta = 0$ was allowed in \Cref{l: convergence stochastic first}, it leads to unfeasible condition $0 \le \kappa < 0$ in \Cref{l: convergence stochastic second} and has to be excluded.

\subsubsection{From the convergence of a subsequence to the convergence of sequence}

\Cref{l: convergence stochastic second} only provides that $\inf_{t \ge 0} \norm{\nabla \q J(z^t,v^t)}_2^2 = 0$, which is equivalent to existence of a subsequence $\{ t_k \}_{k \ge 0}$  such that $\nabla \q J(z^{t_k}, v^{t_k})$ vanishes as $k \to \infty$, see Corollary 3.9 of \cite{Melnyk.2022b}. In order to get that $\nabla \q J(z^{t}, v^{t})$ as $t \to \infty$, we follow the arguments of Theorem 2 in \cite{Orabona.2020}. The cornerstone of their argumentation is the following lemma.

\begin{lemma}[{\cite[Lemma 1]{Orabona.2020}}]\label{l: convergence to 0}
Let $\{\xi_t\}_{t \ge 0}$ and $\{\mu_t\}_{t \ge 0}$ be two nonnegative sequences and $\{ q^t \}_{t \ge 0}$ be a sequence in $\bb C^d$. Assume that 
\[
\sum_{t = 0}^{\infty} \mu_t \xi_t^p < \infty, \quad \sum_{t = 0}^{\infty} \mu_t = \infty \quad \text{ and } \quad \norm{ \sum_{t = 0}^{\infty} \mu_t q^t } < \infty,
\]   
for some $p \ge 1$ and some norm $\norm{\cdot}$. If there exists $L \ge 0$ such that for all $t, k \ge 0$ the inequality
\[
| \xi_{t + k} - \xi_t | \le L \sum_{s = t}^{t + k -1}  \mu_s  \xi_s + L \norm{ \sum_{s = t}^{t + k -1}  \mu_s q^s }
\]
holds, then $\lim_{t \to \infty} \xi_{t} = 0$.
\end{lemma}

Applying \Cref{l: convergence to 0} for $\xi_t := \norm{\nabla_z \q J(z^t,v^t)}_2 + \norm{\nabla_v \q J(z^t,v^t)}_2$ gives the desired convergence. Note that in \cite{Orabona.2020}, $L$ is the Lipschitz constant of the gradient. However, $\q J$ only has locally Lipschitz continuous gradient as stated by the next lemma.
\begin{lemma}\label{eq: Lipschitz continuity}
Let $\varepsilon > 0$. For arbitrary $z_1,z_2,v_1,v_2 \in \bb C^d$, we have
\begin{align*}
& \norm{ \nabla \q J(z_1,v_1) - \nabla \q J(z_2,v_2) }_2 \\
& \quad \quad \quad \quad \le \sqrt{ 2 L^2(z_1,z_2,v_1,v_2) + 2 \max\{\alpha_T^2, \beta_T^2 \}} [\norm{ z_1 -  z_2 }_2^2 + \norm{ v_1 -  v_2 }_2^2]^{1/2},
\end{align*}
with
\[
L(z_1,z_2,v_1,v_2) := d [ \norm{y/d}_1^{1/2} + \max\{\tfrac{5}{4}, \norm{y + \varepsilon}_{\infty}^{1/2} \varepsilon^{-1/2} - \tfrac{3}{4}\} [\norm{z_1}_2^2 + \norm{z_2}_2^2 + \norm{v_1}_2^2 + \norm{v_2}_2^2 ] ].
\]
\end{lemma}   
\begin{proof}
We start by showing the local Lischitz continuity of $\nabla \q L_{\varepsilon}$. 
Denote by $\frac{\partial}{\partial \RE}$ and $\frac{\partial}{\partial \IM}$ standard derivatives with respect to real and imaginary parts of the argument of a function. Moreover, let $\nabla_{\RE,\IM}^2$ be a Hessian of the function with respect to these derivatives. By computations in the proof of Lemma 2.1 on p.17 of \cite{Melnyk.2022} it holds that 
\[
\norm{ \nabla \q L_\varepsilon(z_1,v_1) - \nabla \q L_\varepsilon(z_2,v_2) }_2
= \frac{1}{2} \norm{ 
{
\renewcommand\arraystretch{1.25}
\begin{bmatrix}
(\frac{\partial \q L_\varepsilon}{\partial \RE})^T (z_1,v_1)  \\
(\frac{\partial \q L_\varepsilon}{\partial \IM })^T (z_1,v_1)
\end{bmatrix}
 - 
\begin{bmatrix}
(\frac{\partial \q L_\varepsilon}{\partial \RE})^T (z_2,v_2)  \\
(\frac{\partial \q L_\varepsilon}{\partial \IM })^T (z_2,v_2)
\end{bmatrix}
}
}_2.
\]
Furthermore, by the fundamental theorem of calculus
\begin{align*}
& \norm{
{
\renewcommand\arraystretch{1.25}
\begin{bmatrix}
(\frac{\partial \q L_\varepsilon}{\partial \RE})^T (z_1,v_1)  \\
(\frac{\partial \q L_\varepsilon}{\partial \IM })^T (z_1,v_1)
\end{bmatrix}
 - 
\begin{bmatrix}
(\frac{\partial \q L_\varepsilon}{\partial \RE})^T (z_2,v_2)  \\
(\frac{\partial \q L_\varepsilon}{\partial \IM })^T (z_2,v_2)
\end{bmatrix}
}
}_2 \\
& \quad \quad =
\norm{ \left( \int_0^1 \nabla_{\RE,\IM}^2 \q L_\varepsilon (z_1 + t(z_2 - z_1), v_1 + t(v_2 - v_1) ) dt \right) \begin{bmatrix} \RE(z_2 - z_1) \\ \RE(v_2 - v_1) \\ \IM(z_2 - z_1) \\ \IM(v_2 - v_1) \end{bmatrix} }_2 \\
& \quad \quad \le
\norm{ \left( \int_0^1 \nabla_{\RE,\IM}^2 \q L_\varepsilon ( (1-t) z_1 + t z_2), (1-t) v_1 + t v_2 ) dt \right) }
\norm{\begin{bmatrix} z_2 - z_1 \\ v_2 - v_1 \end{bmatrix} }_2 \\
& \quad \quad \le
\int_0^1 \norm{ \nabla_{\RE,\IM}^2 \q L_\varepsilon ((1-t) z_1 + t z_2, (1-t) v_1 + t v_2 )} dt
\norm{\begin{bmatrix} z_2 - z_1 \\ v_2 - v_1 \end{bmatrix} }_2,
\end{align*} 
where $\norm{ \cdot }$ denotes the spectral norm of the matrix. The calculations on page 16 of \cite{Melnyk.2022} yield the following formula for the spectral norm 
\[
\norm{ \nabla_{\RE,\IM}^2 \q L_\varepsilon (z, v)}
= \max_{s \in \bb C^{2d}, \norm{s}_2 = 1} \left| \begin{bmatrix} s \\ \conj s \end{bmatrix}^*
\nabla^2 \q L_\varepsilon(z,v)
\begin{bmatrix} s \\ \conj s \end{bmatrix} \right|.
\]
Thus, we have
\begin{align}
&\norm{ \nabla \q L_\varepsilon(z_1,v_1) - \nabla \q L_\varepsilon(z_2,v_2) }_2 \label{eq: Lipschitz tech 1}\\
& \quad \quad  = \frac{1}{2} \int_0^1
\max_{s \in \bb C^{2d}, \norm{s}_2 = 1} \left| \begin{bmatrix} s \\ \conj s \end{bmatrix}^*
\nabla^2 \q L_\varepsilon((1-t) z_1 + t z_2, (1-t) v_1 + t v_2)
\begin{bmatrix} s \\ \conj s \end{bmatrix} \right| dt
\norm{\begin{bmatrix} z_2 - z_1 \\ v_2 - v_1 \end{bmatrix} }_2. \nonumber 
\end{align}
Let $s = (u,h) \in \bb C^d$ be arbitrary. Recall that by \Cref{l: Hessian bound}, we have
\[
\left|
\begin{bmatrix} s \\ \conj s \end{bmatrix}^*
\nabla^2 \q L_\varepsilon(z, v)
\begin{bmatrix} s \\ \conj s \end{bmatrix} \right|
\le 2 L_\varepsilon \sum_{j=1}^m |u^T Q_j v + z^T Q_j h |^2 + 4\left[ \q L_\varepsilon(z,v) \right]^{1/2} \cdot \left[ \sum_{j=1}^m |u^T Q_j h|^2 \right]^{1/2},
\]
with $L_\varepsilon := \max\{1, \norm{y + \varepsilon}_{\infty}^{1/2} \varepsilon^{-1/2} - 1\}$. Using \eqref{eq: bilinear norm bound} and \eqref{eq: L bound}, it is further bounded  as
\begin{align*}
& \left| \begin{bmatrix} s \\ \conj s \end{bmatrix}^*
\nabla^2 \q L_\varepsilon(z, v)
\begin{bmatrix} s \\ \conj s \end{bmatrix} \right|
\le 4 L_\varepsilon  \sum_{j=1}^m [ |u^T Q_j v|^2 + |z^T Q_j h |^2] + 4\left[ \q L_\varepsilon(z,v) \right]^{1/2} \cdot \left[ \sum_{j=1}^m |u^T Q_j h|^2 \right]^{1/2} \\
& \quad \quad \quad \quad \le 4 d L_\varepsilon [ \norm{u}_2^2 \norm{v}_2^2 + \norm{z}_2^2 \norm{h}_2^2] + 4\left[ d \norm{z}_2^2 \norm{v}_2^2 + \norm{y}_1 \right]^{1/2} d^{1/2} \norm{u}_2 \norm{h}_2 \\
& \quad \quad \quad \quad \le 4 d L_\varepsilon [ \norm{u}_2^2 + \norm{h}_2^2] \cdot[\norm{z}_2^2 + \norm{v}_2^2 ] + 2 d \left[ \norm{z}_2 \norm{v}_2 + \norm{y/d}_1^{1/2} \right] \cdot [\norm{u}_2^2 + \norm{h}_2^2] \\
%& \quad \quad \quad \quad \le [4 d L_\varepsilon [\norm{v}_2^2 + \norm{z}_2^2 ] + 2 d \left[ \norm{z}_2 \norm{v}_2 + \norm{y/d}_1^{1/2} \right]  ] \cdot \norm{s}_2^2 \\
& \quad \quad \quad \quad \le d \left[ (4 L_\varepsilon + 1) [\norm{z}_2^2 + \norm{v}_2^2 ] + 2 \norm{y/d}_1^{1/2} \right] \cdot \norm{s}_2^2.
\end{align*}
Therefore, we obtain 
\[
\max_{s \in \bb C^{2d}, \norm{s}_2 = 1}
\left|
\begin{bmatrix} s \\ \conj s \end{bmatrix}^*
\nabla^2 \q L_\varepsilon(z, v)
\begin{bmatrix} s \\ \conj s \end{bmatrix}
\right| \le d \left[(4 L_\varepsilon + 1) [\norm{z}_2^2 + \norm{v}_2^2 ] + 2 \norm{y/d}_1^{1/2} \right].
\]
Next, we substitute $z = (1-t)z_1 + t z_2$ and $v = (1-t)v_1 + t v_2$, apply convexity of $\norm{\cdot}_2^2$ and integrate for $t \in [0,1]$, 
\begin{align*}
& \int_0^1 d [(4 L_\varepsilon + 1) [\norm{(1- t) z_1 + t z_2 }_2^2 + \norm{(1-t) v_1 + t v_2}_2^2 ] + 2 \norm{y/d}_1^{1/2} ] dt \\
& \quad \le \int_0^1 d [(4 L_\varepsilon + 1) [(1- t) \norm{z_1}_2^2 + t \norm{z_2}_2^2 + (1-t) \norm{v_1}_2^2 + t \norm{v_2}_2^2 ] + 2 \norm{y/d}_1^{1/2} ] dt \\
& \quad = d [ \tfrac{1}{2} (4 L_\varepsilon + 1) [\norm{z_1}_2^2 + \norm{z_2}_2^2 + \norm{v_1}_2^2 + \norm{v_2}_2^2 ] + 2 \norm{y/d}_1^{1/2} ].
\end{align*}
By combining the above inequalities with \eqref{eq: Lipschitz tech 1}, we obtain
\begin{align*}
&
\norm{ \nabla \q L_\varepsilon(z_1,v_1) - \nabla \q L_\varepsilon(z_2,v_2) }_2 \\
& \quad  \le d \left[  (L_\varepsilon + \tfrac{1}{4}) [\norm{z_1}_2^2 + \norm{z_2}_2^2 + \norm{v_1}_2^2 + \norm{v_2}_2^2 ] + \norm{y/d}_1^{1/2} \right]
\norm{\begin{bmatrix} z_2 - z_1 \\ v_2 - v_1 \end{bmatrix} }_2,
\end{align*}
with the coefficient being precisely $L(z_1,z_2,v_1,v_2)$ defined in the statement of the lemma. Turning to $\q J$, we have
\begin{align*}
\norm{ \nabla_z \q J(z_1,v_1) - \nabla_z \q J(z_2,v_2) }_2^2
& = \norm{ \nabla_z \q L_\varepsilon(z_1,v_1)  + \alpha_T z_1 - \nabla \q L_\varepsilon(z_2,v_2) - \alpha_T z_2 }_2^2 \\
& \le 2 \norm{ \nabla_z \q L_\varepsilon(z_1,v_1)  - \nabla_z \q L_\varepsilon(z_2,v_2)}_2^2 + 
2\alpha_T^2 \norm{ z_1 -  z_2 }_2^2.
\end{align*}
and, analogously,
\[
\norm{ \nabla_v \q J(z_1,v_1) - \nabla_v \q J(z_2,v_2) }_2^2
\le 2 \norm{ \nabla_v \q L_\varepsilon(z_1,v_1)  - \nabla_v \q L_\varepsilon(z_2,v_2)}_2^2 +  2\beta_T^2 \norm{ v_1 -  v_2 }_2^2.
\]
Combined, it gives
\begin{align*}
& \norm{ \nabla \q J(z_1,v_1) - \nabla \q J(z_2,v_2) }_2^2 \\
& \quad \quad \quad \quad \le  2 \norm{ \nabla \q L_\varepsilon(z_1,v_1)  - \nabla \q L_\varepsilon(z_2,v_2)}_2^2 + 2\alpha_T^2 \norm{ z_1 -  z_2 }_2^2 +  2\beta_T^2 \norm{ v_1 -  v_2 }_2^2 \\
& \quad \quad \quad \quad \le (2 L^2(z_1,z_2,v_1,v_2) + 2 \max\{\alpha_T^2, \beta_T^2 \}) [\norm{ z_1 -  z_2 }_2^2 + \norm{ v_1 -  v_2 }_2^2]. 
\end{align*}
\end{proof}

However, despite the gradient being Lipschitz continuous only locally, we are still able to show that the gradients vanish. 

\begin{lemma}\label{l: convergence stochastic third}
Let $\varepsilon, \alpha_T, \beta_T > 0$, $0 < \theta < 1$, $0 \le \kappa < \theta/(1 + \theta)$. Consider the step sizes $\mu_t = \mu  \cdot \mu_t^{\max}$ and $\nu_t = \nu \cdot \mu_t^{\max}$ for some $0 < \mu,\nu \le 1$ and $\mu_t^{\max}$ as in \eqref{eq: mu max}. Then, we have
\[
\lim_{t \to \infty} \norm{\nabla \q J(z^t,v^t)}_2^2  = 0 \quad \text{a.s.}
\]
\end{lemma}
\begin{proof}
Let us start by establishing the inequality in \Cref{l: convergence to 0} for $\xi_t := \norm{\nabla_z \q J(z^t,v^t)}_2 + \norm{\nabla_v \q J(z^t,v^t)}_2$. For any $t \ge 0$ and $k > 0$, by the reverse- and triangle inequalities and \Cref{eq: Lipschitz continuity}, we have
\begin{align*}
& \left| \, \norm{ \nabla_z \q J(z^{t+k},v^{t+k}) }_2 + \norm{ \nabla_v \q J(z^{t+k},v^{t+k}) }_2 -  \norm{ \nabla_z \q J(z^t,v^t) }_2 -  \norm{ \nabla_v \q J(z^t,v^t) }_2 \, \right| \\
&\quad \quad \quad \quad  \le \norm{ \nabla_z \q J(z^{t+k},v^{t+k}) - \nabla_z \q J(z^t,v^t) }_2 + \norm{ \nabla_v \q J(z^{t+k},v^{t+k}) - \nabla_v \q J(z^t,v^t) }_2 \\
& \quad \quad \quad \quad \le \sqrt 2 \norm{ \nabla \q J(z^{t+k},v^{t+k}) - \nabla \q J(z^t,v^t) }_2  \\
& \quad \quad \quad \quad \le 2 \sqrt{ L^2(z^{t+k},z^t,v^{t+k},v^t) + \max\{\alpha_T^2, \beta_T^2 \}} [\norm{ z^{t+k} -  z^t }_2^2 + \norm{ v^{t+k} -  v^t }_2^2]^{1/2} \\
& \quad \quad \quad \quad \le 2 \sqrt{ L^2(z^{t+k},z^t,v^{t+k},v^t) + \max\{\alpha_T^2, \beta_T^2 \}} [\norm{ z^{t+k} -  z^t }_2 + \norm{ v^{t+k} -  v^t }_2].
\end{align*}
Next, we show that the square root is bounded from above for all $t$ and $k$. Since the assumptions of \Cref{l: convergence stochastic first} are satisfied there exists a random variable $\q J^{\sup} < +\infty$ a.s. and by the arguments in the proof of \Cref{l: convergence stochastic second} (see \eqref{eq: norm bounded stochastic}) we have 
\[
\norm{z^t}_2^2 \le \alpha_T^{-1} \q J^{\sup} 
\quad \text{and} \quad
\norm{v^t}_2^2 \le \beta_T^{-1} \q J^{\sup}
\quad \text{for all } t. 
\]
Consequently, the constant $L(z^{t+k},z^t,v^{t+k},v^t)$ from \Cref{eq: Lipschitz continuity} is bounded by
\[
L(z^{t+k},z^t,v^{t+k},v^t) \le d [ 2 \q J^{\sup} \max\{\tfrac{5}{4}, \norm{y + \varepsilon}_{\infty}^{1/2} \varepsilon^{-1/2} - \tfrac{3}{4}\} [\alpha_T^{-1} + \beta_T^{-1}] + \norm{y/d}_1^{1/2} ].
\]
Let us set $L^{\sup} > 0$ as
\[
(L^{\sup})^2 := 4d^2 [ 2 \q J^{\sup} \max\{\tfrac{5}{4}, \norm{y + \varepsilon}_{\infty}^{1/2} \varepsilon^{-1/2} - \tfrac{3}{4}\} [\alpha_T^{-1} + \beta_T^{-1}] + \norm{y/d}_1^{1/2} ]^2 + 4 \max\{\alpha_T^2, \beta_T^2 \},
\]
which leads to
\[
\left| \norm{ \nabla \q J(z^{t+k},v^{t+k}) }_2 -  \norm{ \nabla \q J(z^t,v^t) }_2 \right|
\le L^{\sup} \left[ \norm{ z^{t+k} -  z^t }_2 + \norm{ v^{t+k} -  v^t }_2 \right].
\]
Furthermore, by construction
\begin{align*}
\norm{z^{t+k} -  z^t }_2
& = \norm{ \sum_{s = t}^{t+k-1} z^{s+1} - z^{s} }_2
=  \norm{ \sum_{s = t}^{t+k-1} \mu_s  g_z(z^s,v^s) }_2 \\
& \le \sum_{s = t}^{t+k-1} \mu_s \norm{ \nabla_z \q J(z^s,v^s) }_2 
+ \norm{ \sum_{s = t}^{t+k-1} \mu_s \left[ g_z(z^s,v^s) -  \nabla_z \q J(z^s,v^s) \right] }_2,
\end{align*}
and, analogously,
\[
\norm{v^{t+k} -  v^t }_2
\le \sum_{s = t}^{t+k-1} \nu_s \norm{ \nabla_v \q J(z^s,v^s) }_2 
+\norm{ \sum_{s = t}^{t+k-1} \nu_s \left[ g_v(z^s,v^s) -  \nabla_v \q J(z^s,v^s) \right] }_2.
\]
Define $q_s$, $s \ge 0$, for \Cref{l: convergence to 0} as 
\begin{equation}\label{eq: martingale summand}
q_s := \begin{bmatrix}
\mu g_z(z^s,v^s) -  \mu \nabla_z \q J(z^s,v^s)  \\ 
\nu  g_v(z^s,v^s) -  \nu \nabla_v \q J(z^s,v^s) 
\end{bmatrix} \in \bb C^{2d}.
\end{equation}
The corresponding norm in \Cref{l: convergence to 0} is the Euclidean norm and by $(a +b)^2 \le 2 a^2 + 2 b^2$ we have
\begin{align*}
& \norm{ \sum_{s = t}^{t+k-1} \mu_s \left[ g_z(z^s,v^s) -  \nabla_z \q J(z^s,v^s) \right] }_2 \\
& \quad + \norm{ \sum_{s = t}^{t+k-1} \nu_s \left[ g_v(z^s,v^s) -  \nabla_v \q J(z^s,v^s) \right] }_2
\le \sqrt 2 \norm{\sum_{s = t}^{t+k-1} \mu^{\max}_s q_s }_2.
\end{align*}
Now, we combine the last few steps together and use that $\mu_t = \mu \cdot \mu^{\max}_t$, $\nu_t = \nu \cdot \mu^{\max}_t$ and $\mu, \nu \le 1$,
\begin{align*}
\left| \xi_{t+k} -  \xi_t \right|
& \le  L^{\sup} \sum_{s = t}^{t+k-1}\left[\mu_s \norm{ \nabla_z \q J(z^s,v^s) }_2  + \nu_s \norm{ \nabla_v \q J(z^s,v^s) }_2 \right] + \sqrt{2} L^{\sup} \sum_{s = t}^{t+k-1} \mu^{\max}_s q_s \\  
& \le \sqrt{2} L^{\sup} \sum_{s = t}^{t+k-1} \mu^{\max}_s \xi_s  + L^{\sup} \left| \sum_{s = t}^{t+k-1} \mu^{\max}_s q_s \right|.
\end{align*}
Let us now show that the rest of the conditions in \Cref{l: convergence to 0} hold for the constructed sequences. For the first condition with $p=2$, by \Cref{l: convergence stochastic first}, we get
\begin{align*}
\sum_{t \ge 0} \mu^{\max}_t \xi_t^2
& = \sum_{t \ge 0} \mu^{\max}_t \left( \norm{ \nabla_z \q J(z^t,v^t) }_2  + \norm{ \nabla_v \q J(z^t,v^t) }_2 \right)^2\\
& \le \frac{2}{\min\{\mu,\nu\}} \sum_{t \ge 0} \mu_t \norm{ \nabla_z \q J(z^t,v^t) }_2^2  + \nu_t \norm{ \nabla_v \q J(z^t,v^t) }_2^2 < \infty \quad \text{a.s.}
\end{align*}
The series $\sum_{t \ge 0} \mu_t^{\max}$ diverges a.s. by the comparison test and the lower bound \eqref{eq: mu diverges}. For the last condition, we consider a sequence
\[
M_0 := 0, \quad M_t := \sum_{s=0}^{t-1} \mu_s^{\max} q_s, \quad t \ge 1.
\] 
By construction, it is adapted to the filtration $\{\q F_t\}$ and by construction of $\mu_t^{\max}$ and \Cref{l: stochastic gradient properties 2} admits 
\[
\bb E[\, M_{t+1} \, | \, \q F_t \,] 
% = M_{t} + \bb E[\, \mu_t^{\max} q_t \, | \, \q F_t \,] 
= M_{t} + \mu_t^{\max}
\begin{bmatrix}
\mu \bb E[\, g_z(z^s,v^s) \, | \, \q F_t \,]  -  \mu \nabla_z \q J(z^s,v^s)  \\ 
\nu \bb E[\,  g_v(z^s,v^s)\, | \, \q F_t \,]  -  \nu \nabla_v \q J(z^s,v^s) 
\end{bmatrix}
= M_{t}.
\]
Hence, $M_t$ is marginal. Its increments $\mu_t^{\max} q_t$ satisfy
\begin{align*}
& \bb E[\, \norm{\mu_t^{\max} q_t}_2^2 \, | \, \q F_t \,] \\
%& = \bb E[  \,  (\mu_t \norm{ g_z(z^t,v^t) -  \nabla_z \q J(z^t,v^t) }_2 + \nu_t \norm{ g_v(z^t,v^t) -  \nabla_v \q J(z^t,v^t) }_2)^2 \, | \, \q F_t \,] \\
& \quad \quad = \bb E[\, \mu_t^2 \norm{ g_z(z^t,v^t) -  \nabla_z \q J(z^t,v^t) }_2^2 + \nu_t^2 \norm{ g_v(z^t,v^t) -  \nabla_v \q J(z^t,v^t) }_2^2 \, | \, \q F_t \,] \\
& \quad \quad = \mu_t^2 \bb E[\,  \norm{ g_z(z^t,v^t) -  \nabla_z \q J(z^t,v^t) }_2^2 \, | \, \q F_t \,] + \nu_t^2 \bb E[\, \norm{ g_v(z^t,v^t) -  \nabla_v \q J(z^t,v^t) }_2^2 \, | \, \q F_t \,].
\end{align*}
Let us expand the first expectation. By \Cref{l: stochastic gradient properties 2}, we get  
\begin{align*}
\bb E[\,  \norm{ g_z(z^t,v^t) -  \nabla_z \q J(z^t,v^t) }_2^2 \, | \, \q F_t \,] 
& = \bb E[\,  \norm{ g_z(z^t,v^t)}_2^2 \, | \, \q F_t \,] -  \norm{\nabla_z \q J(z^t,v^t) }_2^2 \\
& \le \gamma_t^z [\q J(z^t, v^t) - \q J^{\inf}] - K^{-1} \norm{\nabla_z \q J(z^t, v^t)}_2^2 + \delta_t^z,
\end{align*}
and the second expectation is analogous. Thus, the tower property yields
\begin{align*}
\bb E \norm{ \mu_t^{\max} q_t}_2^2 
& = \bb E [ \, \bb E[\, \norm{ \mu_t^{\max} q_t }_2^2 \, | \, \q F_t \,] \, ]  
\le 2 \bb E [ (\mu_t^2 \gamma_t^z + \nu_t^2 \gamma_t^v ) [\q J(z^t, v^t) - \q J^{\inf}] \\
& \quad - K^{-1} [\mu_t^2 \norm{\nabla_z \q J(z^t, v^t)}_2^2 + \nu_t^2 \norm{\nabla_v \q J(z^t, v^t)}_2^2 ] + (\mu_t^2 \delta_t^z + \nu_t^2 \delta_t^v)].
\end{align*}
Consequently, by \cite[Problem 13.27]{Athreya.2006}, the expected squared norm of $M_t$ is bounded by
\begin{align}
\bb E \norm{M_t}_2^2  
& = \sum_{s=0}^{t-1} \bb E \norm{\mu_s^{\max} q_s}_2^2 
\le 2 \sum_{s=0}^{t-1}\bb E \left[ (\mu_t^2 \gamma_t^z + \nu_t^2 \gamma_t^v ) [\q J(z^t, v^t) - \q J^{\inf} ] \right] \label{eq: Mt expected squared norm}\\
& - 2 K^{-1} \sum_{s=0}^{t-1}\bb E \left[ \mu_t^2 \norm{\nabla_z \q J(z^t, v^t)}_2^2 + \nu_t^2 \norm{\nabla_v \q J(z^t, v^t)}_2^2 \right]
+ 2  \sum_{s=0}^{t-1}\bb E[ \mu_t^2 \delta_t + \nu_t^2 \rho_t]. \nonumber 
\end{align} 
We recall that by construction and \eqref{eq: mu first part bound}, 
\[
\mu_t, \nu_t 
\le \mu_t^{\max} 
\le  [3 \sqrt d \norm{y}_1^{1/2}  + 3\max\{\alpha_T,\beta_T\}]^{- \frac{1}{1-\theta}} =: \mu_{\sup}.
\]
Also, recall the definition \eqref{eq: alpha mu bound} of $\eta_t$ from the proof of \Cref{l: convergence stochastic first}. Then, by the inequality \eqref{eq: alpha mu bound} we have
\begin{align*}
\sum_{s = 0}^{t-1}\bb E \left[ (\mu_t^2 \gamma_t^z + \nu_t^2 \gamma_t^v ) [\q J(z^t, v^t) - \q J^{\inf} ] \right]
& \le \sum_{s = 0}^{t-1} \bb E \left[ \mu_{\sup}^{1 - \theta} (\mu_t^{1+\theta} \gamma_t^z + \nu_t^{1+\theta} \gamma_t^v ) [\q J^{\sup} - \q J^{\inf}] \right] \\
& \le \mu_{\sup}^{1 - \theta} [\bb E \q J^{\sup} - \q J^{\inf}] \sum_{s = 0}^{t-1} \eta_s 
\end{align*} 
and, similarly, by \eqref{eq: delta mu bound}, the inequality
\begin{align*}
\sum_{s = 0}^{t-1} \bb E [ \mu_t^2 \delta_t + \nu_t^2 \rho_t] 
& \le \mu_{\sup}^{1 - \theta} [ \q J^{\inf} - \sum_{r \in \q R} \inf_{z,v \in \bb C^d} \q J_{r}(z,v)] \sum_{s = 0}^{t-1} \eta_s,
\end{align*} 
holds. The second summand in \eqref{eq: Mt expected squared norm} can be bounded from above by zero. This gives
\[
\bb E \norm{M_t}_2^2  \le 2 \mu_{\sup}^{1 - \theta}[\bb E \q J^{\sup} - \sum_{r \in \q R} \inf_{z,v \in \bb C^d} \q J_{r}(z,v)] \sum_{s = 0}^{t-1} \eta_s.
\]
Since $\q J^{\sup} < +\infty$ a.s. and the series $\sum_{s \ge 0} \eta_s$ is convergent, $\bb E \norm{M_t}_2^2 < \infty$. By the Jensen's inequality \cite[Proposition 6.2.6]{Athreya.2006}, each of its coordinates $(M_t)_j$, $j =1, \ldots, d$, admits 
\[
\sup_{t \ge 0} \bb E[ \max\{(M_t)_j , 0\}]
\le \sup_{t \ge 0} \bb E |(M_t)_j |
\le \sup_{t \ge 0} \bb E \norm{M_t}_2
\le \sup_{t \ge 0} \sqrt{\bb E \norm{M_t}_2^2} 
< \infty.
\]
Thus, by \cite[Theorem 13.3.2]{Athreya.2006}, there exists a random vector $M$ such that $\bb E |M_j| < \infty$ for all $j =1, \ldots, d$, which implies $M_j < \infty$ almost surely. Therefore, we obtain the last condition of \Cref{l: convergence to 0},
\[
\norm{ \sum_{s=0}^\infty \mu_s q_s }_2
= \norm{ \lim_{t \to \infty} M_t }_2 
= \norm{M}_2 
= \left[ \sum_{j =1}^d |M_j|^2 \right]^{1/2} < \infty \quad \text{a.s}.
\]
Hence, by \Cref{l: convergence to 0}, $\lim_{t \to \infty} \xi_t = 0$ a.s. This gives 
\begin{align*}
0 & \le \lim_{t \to \infty} \norm{\nabla_z \q J(z^t,v^t)}_2^2 + \norm{\nabla_v \q J(z^t,v^t)}_2^2 \\
& \le \lim_{t \to \infty} [\norm{\nabla_z \q J(z^t,v^t)}_2 + \norm{\nabla_v \q J(z^t,v^t)}_2]^2 
= \lim_{t \to \infty} \xi_t^2
= 0.
\end{align*} 
\end{proof}

\begin{remark}
If indices $r^{t,k}$, $k=1,\ldots K$ are sampled from some other sampling scheme with $\rho >1$ in \Cref{l: stochastic gradient properties 2}, we can no longer discard the second sum in \eqref{eq: Mt expected squared norm}. Instead, consider the sequence $H_t := \sum_{s=0}^{t-1}\mu_s \norm{\nabla_z \q J(z^s, v^s)}_2^2 + \nu_t \norm{\nabla_v \q J(z^s, v^s)}_2^2$.
Then, the second term in \eqref{eq: Mt expected squared norm} is bounded by
\begin{align*}
\sum_{s \ge 0}\bb E[ \mu_t^2 \norm{\nabla_z \q J(z^t, v^t)}_2^2 + \nu_t^2 \norm{\nabla_v \q J(z^t, v^t)}_2^2]
\le \mu_{\sup} \lim_{t \to \infty} \bb E[H_t ].
\end{align*} 
The sequence $H_t$ is nonnegative, nondecreasing and by \Cref{l: convergence stochastic first} converges to a finite value
$\sum_{s \ge 0} \mu_t \norm{\nabla_z \q J(z^t, v^t)}_2^2 + \nu_t \norm{\nabla_v \q J(z^t, v^t)}_2^2$ a.s. Therefore, the monotone convergence theorem \cite[Theorem 2.3.4]{Athreya.2006} states that there exists a limit 
\[
\lim_{t \to \infty} \bb E H_t = \bb E\left[ \sum_{s \ge 0} \mu_t \norm{\nabla_z \q J(z^t, v^t)}_2^2 + \nu_t \norm{\nabla_v \q J(z^t, v^t)}_2^2 \right] =: \bb E H_\infty,
\]
and $\bb E[ H_\infty ] < \infty$. Consequently, $\bb E \norm{M_t}_2^2$ is finite and the rest of the proof of \Cref{l: convergence stochastic third} remains the same. 
\end{remark} 

Finally, we summarize all results together.
\begin{proof}[Proof of \Cref{thm: convergence stochastic gradient}.]
Each of three claims follow from \Cref{l: convergence stochastic first,l: convergence stochastic second,l: convergence stochastic third} respectively.
\end{proof}

\section{Conclusions}

In this paper, we analyzed the convergence of gradient descent and stochastic gradient descent with the motivation to understand the performance of extended Ptychographic Iterative Engine \cite{Maiden.2009}. While we were able to derive its first convergence guarantees, there are still a lot of open questions left to answer. For instance, the common practical scenario is when $\alpha_t, \beta_t$ are constant. This somewhat links ePIE to stochastic gradient descent with constant step sizes, which current state-of-the-art optimization literature is not able to explain without additional assumptions on the objective $\q J$ \cite{Ding.2019, Dieuleveut.2020,Yu.2021}. Considering the discussion in \Cref{sec: larger step size} we naturally ask ourselves if it is possible to find larger steps sizes for gradient methods without involving too many additional computations. Another interesting direction of future research is to consider regularized PIE (rPIE) \cite{Maiden.2017} and extend our analysis for it. 

\section*{Acknowledgment}
This work was funded by the Helmholtz Association under contracts No.~ZT-I-0025 (Ptychography 4.0), No.~ZT-I-PF-4-018 (AsoftXm), No.~ZT-I-PF-5-28 (EDARTI), No.~ZT-I-PF-4-024 (BRLEMMM). 

\bibliography{ms.bbl} 
%\bibliography{epie_references} 
%\bibliography{stochastic_AF_and_pie_convergence.bbl}
%\appendix

\end{document}